\documentclass[runningheads,a4paper]{llncs}

\usepackage{amssymb}
\usepackage{graphicx}
\usepackage{url}
\usepackage{amsmath}
\newcommand{\keywords}[1]{\par\addvspace\baselineskip
\noindent\keywordname\enspace\ignorespaces#1}

\newcommand{\off}[1]{}  
\def\p{\partial}
\def\IR{\relax{\rm I\kern-.18em R}}
\DeclareMathOperator{\Div}{div}

\urldef{\mailsa}\path| {martin.burger, lina.eckardt}@uni-muenster.de, guy.gilboa@ef.technion.ac.il,moellerm@in.tum.de|  
\urldef{\mailsb}\path| {guy.gilboa@ef.technion.ac.il|  
\begin{document}

\mainmatter

\def\SSVM15SubNumber{61}

\title{Spectral Representations \\ of One-Homogeneous Functionals}  

\titlerunning{Spectral Representations of One-Homogeneous Functionals}
\authorrunning{M. Burger, L. Eckardt, G. Gilboa, M. Moeller}
\author{Martin Burger\inst{1}, Lina Eckardt\inst{1}, Guy Gilboa\inst{2}, Michael Moeller\inst{3}}
\institute{Institute for Computational and Applied Mathematics,
University of M\"unster \and Electrical Engineering Department, Technion – IIT \and Department of Mathematics, Technische Universit\"at M\"unchen}

\maketitle

\begin{abstract}
This paper discusses a generalization of spectral representations related to convex one-homogeneous regularization functionals,  e.g. total variation or $\ell^1$-norms. Those functionals serve as a substitute for a Hilbert space structure (and the related norm) in classical linear spectral transforms, e.g. Fourier and wavelet analysis. We discuss three meaningful definitions of spectral representations by scale space and variational methods and prove that (nonlinear) eigenfunctions of the regularization functionals are indeed atoms in the spectral representation. Moreover, we verify further useful properties related to orthogonality of the decomposition and the Parseval identity.

The spectral transform is motivated by total variation and further developed to higher order variants. Moreover, we show that the approach can recover Fourier analysis as a special case using an appropriate $\ell^1$-type functional and discuss a coupled sparsity example.
\keywords{Nonlinear spectral decomposition, nonlinear eigenfunctions, total variation, convex regularization }
\end{abstract}

\section{Introduction}
Eigenfunction analysis has been used extensively to solve numerous signal processing, computer vision and machine-learning problems
such as segmentation \cite{shi_malik00}, clustering \cite{Ng_Spectral_clustering_2002} and subspace clustering
\cite{Liu_subspace_LRR_2013}, dimensionality reduction \cite{Laplacian_eigenmaps_belkin_niyogi_2003}, and more.
Recently, several attempts were made to generalize some of the properties of the linear setting to a nonlinear setting.
Specifically, singular-value analysis of convex functionals (interpreted as ``ground states") \cite{Benning_Burger_2013} and a spectral representation related
to the total variation (TV) functional \cite{Gilboa_spectv_SIAM_2014} were proposed.

We examine solutions of the following nonlinear eigenvalue problem:
\begin{equation}
\label{eq:ef_problem}
\lambda u  \in \p J(u),
\end{equation}
where $J(u)$  is a convex functional and $\p J(u)$ is its subgradient (precise definitions are given in Section \ref{sec:defs}).
We refer to functions $u$ admitting \eqref{eq:ef_problem} as eigenfunctions with corresponding eigenvalue $\lambda$.

In \cite{Gilboa_spectv_SIAM_2014} a generalization of eigenfunction analysis to the total-variation case was proposed in the following way.
Let $u(t;x)$ be the TV-flow solution \cite{tvFlowAndrea2001} or the gradient descent of the total variation energy $J_{TV}(u)$,  with initial condition $f(x)$:
\begin{equation}
\label{eq:TVflow}
	\partial_t u = - p, \qquad p \in \p J_{TV}(u), \qquad u(t=0)=f(x),
\end{equation}
where
\begin{equation}
	J_{TV}(u) = \sup_{\|\varphi\|_{L^\infty(\Omega)\le 1}} \int_\Omega u \Div \varphi dx,
\end{equation}
with $\varphi \in C^\infty_0$.
The TV spectral transform is defined by
\begin{equation}
\label{eq:phi}
\phi(t;x) := t\partial_{tt}u (t;x),
\end{equation}
where $\partial_{tt}u$ is the second time derivative of the solution $u(t;x)$ of the TV flow \eqref{eq:TVflow}.
For $f(x)$ admitting \eqref{eq:ef_problem}, with a corresponding eigenvalue $\lambda$, one obtains $\phi(t;x)=\delta(t-1/\lambda)f(x)$, where $\delta$ denotes a Dirac delta distribution.
When $f$ is composed of separable eigenfunctions with eigenvalues $\lambda_i$ one obtains through $\phi(t;x)$ a decomposition of the image into its eigenfunctions at $t=1/\lambda_i$.
In the general case, $\phi$ yields a continuum multiscale representation of the image, generalizing structure-texture decomposition
methods like \cite{Meyer[1],Luminita[2],agco06}.
One can reconstruct the original image by:
\begin{equation}
\label{eq:tv_recon}
f(x) = \int_0^\infty \phi(t;x) dt + \bar{f},
\end{equation}
where $\bar{f} = \frac{1}{\Omega}\int_\Omega f(x)dx$.
Given a transfer function $H(t)\in \IR$, image filtering can be performed by:
\begin{equation}
\label{eq:tv_filt}
f_H(x) := \int_0^\infty H(t)\phi(t;x) dt + \bar{f}.
\end{equation}

The spectrum $S(t)$ corresponds to the amplitude of each scale:
\begin{equation}
\label{eq:S}
S(t) := \|\phi(t;x)\|_{L^1(\Omega)} = \int_\Omega |\phi(t;x)|dx.
\end{equation}
\begin{figure}[htb]
\begin{center}
\begin{tabular}{ ccc }
\includegraphics[width=0.33\textwidth]{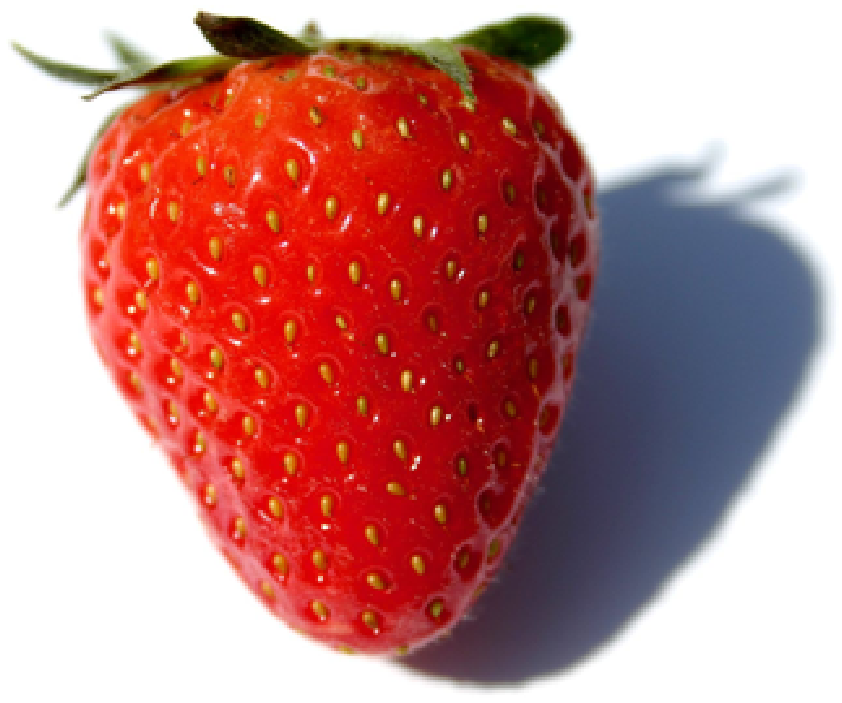}&
\includegraphics[width=0.33\textwidth]{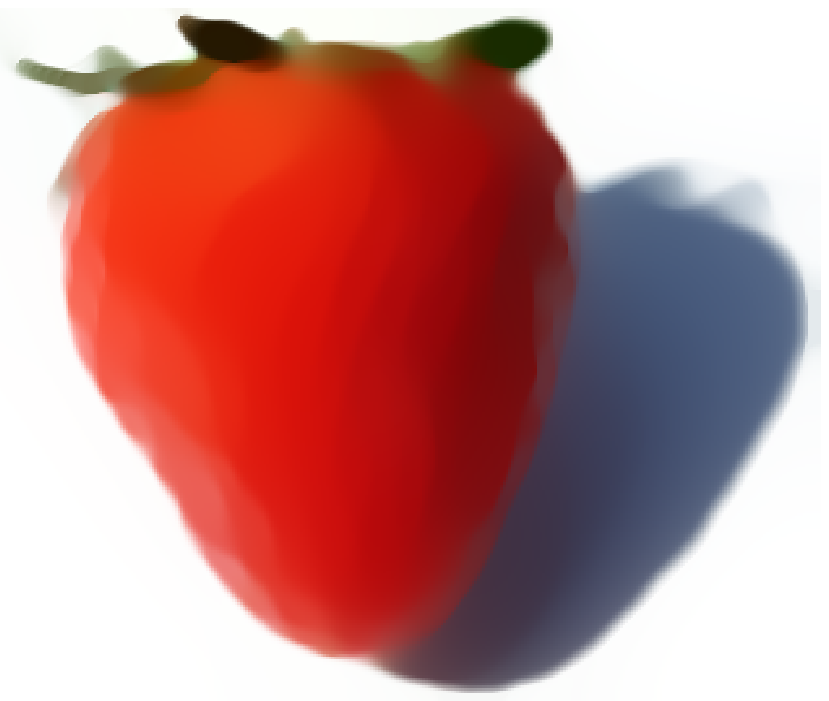}&
\includegraphics[width=0.33\textwidth]{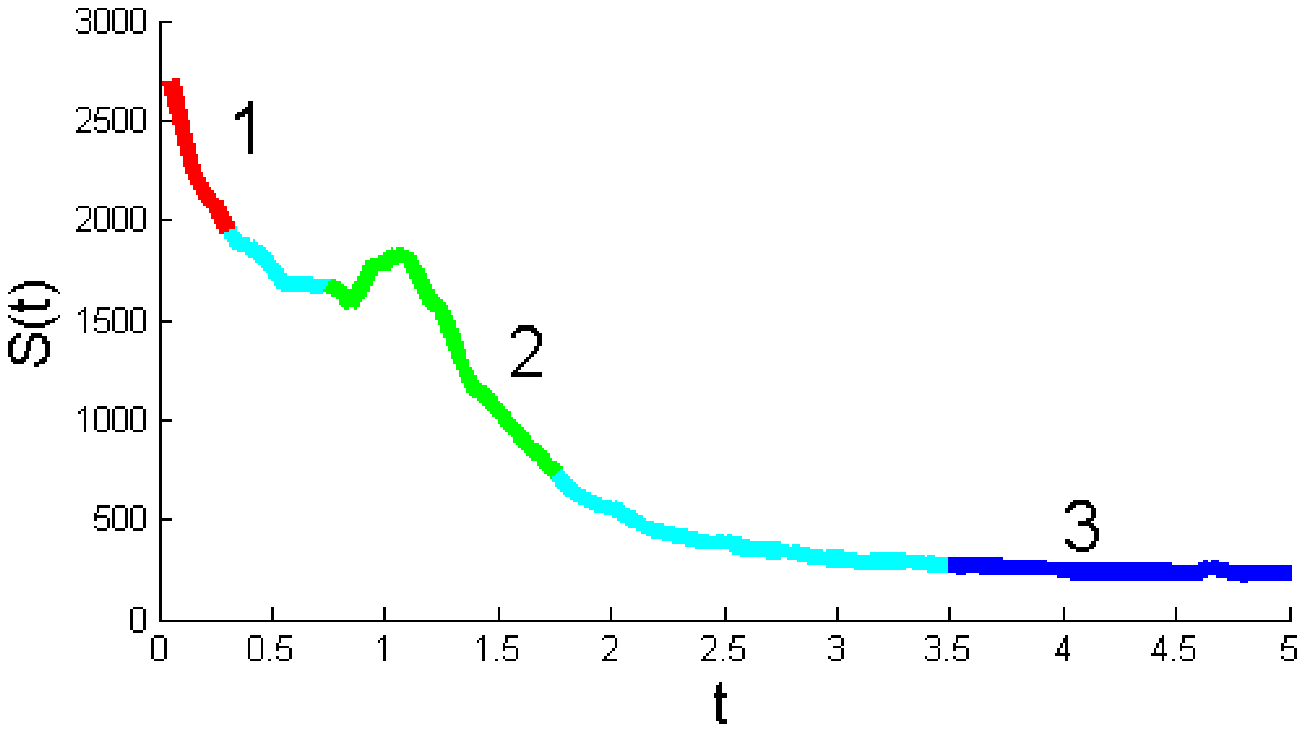}\\
Input $f$ & Low-pass & $S(t)$\\
\includegraphics[width=0.325\textwidth]{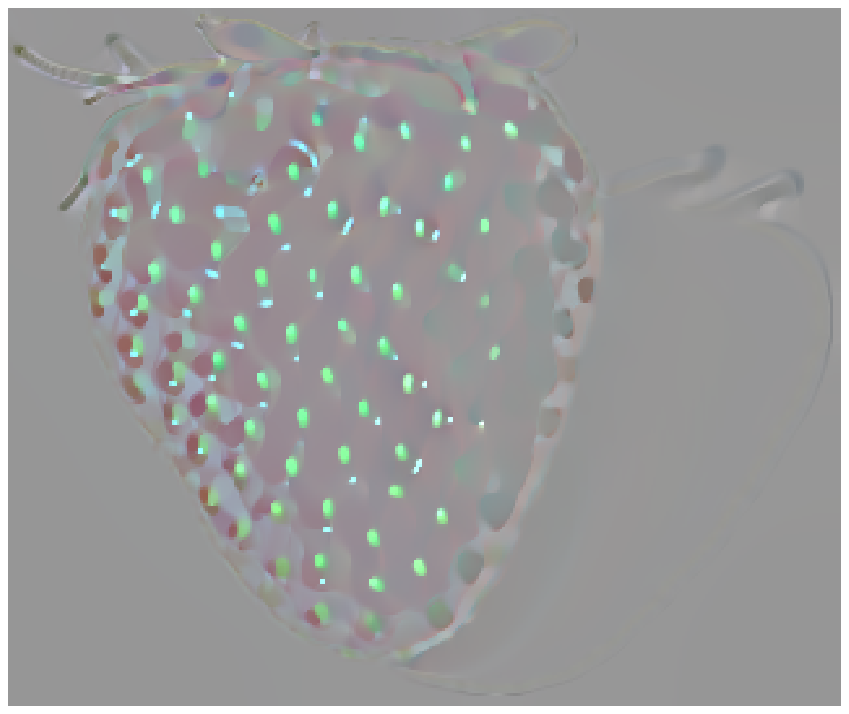}&
\includegraphics[width=0.325\textwidth]{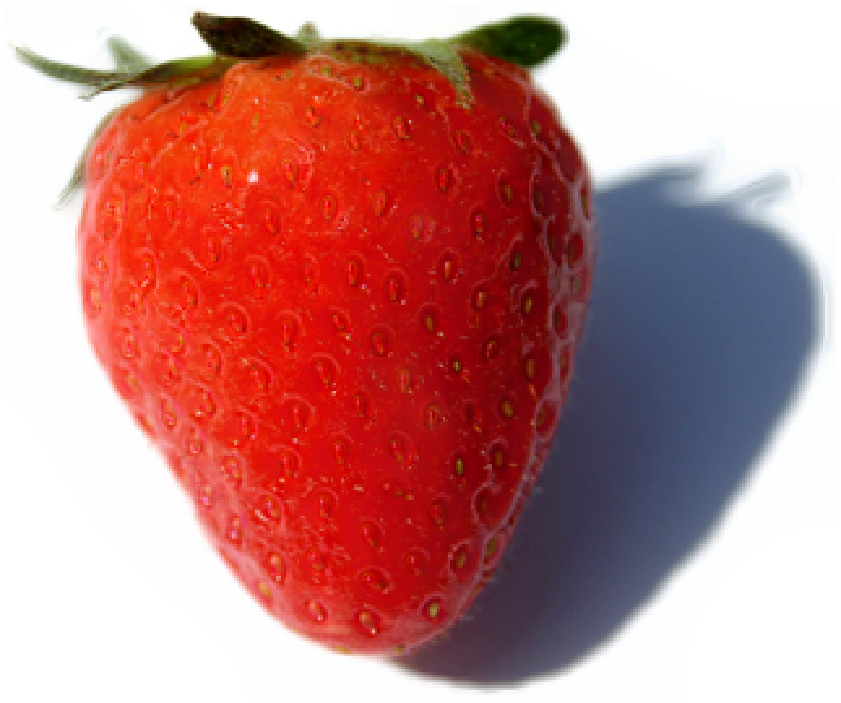}&
\includegraphics[width=0.325\textwidth]{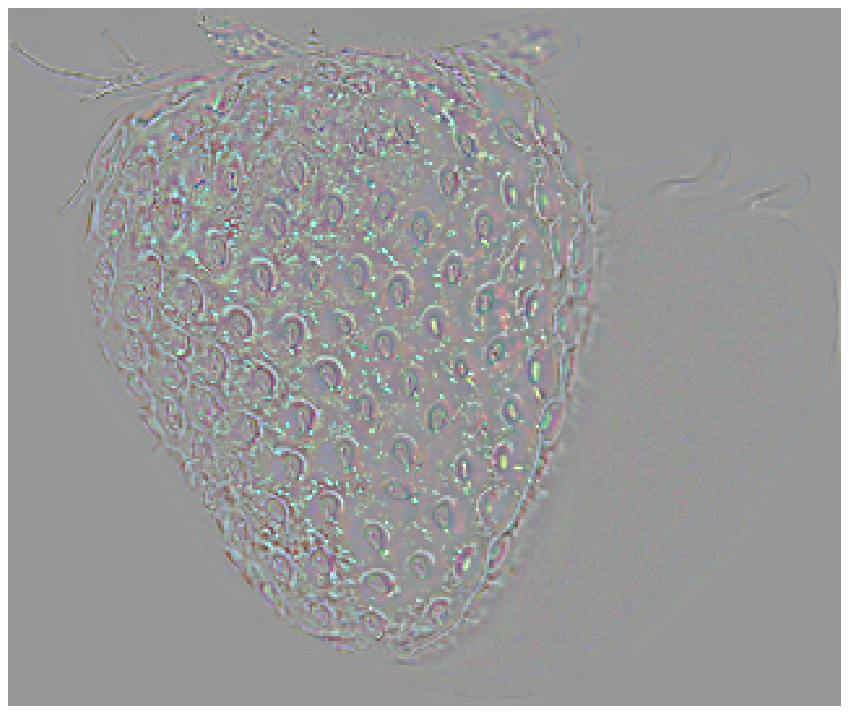}\\
Band-pass & Band-stop & High-pass
\end{tabular}
\caption{Total-variation spectral filtering example. The input image (top left) is decomposed into its $\phi(t)$ components, the corresponding spectrum
$S(t)$ is on the top right. Integration of the $\phi$'s over the $t$ domains 1, 2 and 3 (top right) yields high-pass, band-pass and low-pass filters, respectively. The band-stop filter (bottom middle) is the complement integration domain of region 2.}
\label{fig:strawberry}
\end{center}
\end{figure}
In Fig. \ref{fig:strawberry} an example of spectral TV processing is shown. For example, an (ideal) high-pass filter is defined by $H(t)=1$ for a range
$t \in [0,t_c]$ and 0 otherwise, the filter response is calculated using Eq. \eqref{eq:tv_filt}. Similarly for band-pass and low-pass filters in different time domains (see precise definitions in \cite{Gilboa_spectv_SIAM_2014}).

The analysis of eigenfunctions related to non-quadratic smoothing convex functionals has mainly focused on the TV functional.
An extensive study was conducted in the past decade regarding the TV-flow, its properties like finite extinction time and some analytic solutions
in the case of eigenfunctions, see \cite{tvFlowAndrea2001,andreu2002some,bellettini2002total,Steidl,burger2007inverse_tvflow,discrete_tvflow_2012,tvf_giga2010}.
In \cite{Muller_thesis,Benning_highOrderTV_2013} eigenfunctions related to the total-generalized-variation (TGV) functional \cite{bredies_tgv_2010} are analyzed.

\section{Spectral Representations}
\label{sec:defs}

In the following we generalize the total variation spectral approach in  \cite{Gilboa_SSVM_2013_SpecTV,Gilboa_spectv_SIAM_2014} in two ways. First of all we now consider arbitrary one-homogeneous convex functionals $J: {\cal X} \rightarrow \mathbb{R}^+ \cup \{\infty\}$ defined on Banach space ${\cal X}$ embedded into $L^2(\Omega)$. We assume that the properties of $J$ are such that all differential equations and variational problems considered in the following are well-posed, which can be verified for standard examples such as total variation. Secondly we also consider alternative definitions of the spectral representation than the TV flow used in \cite{Gilboa_SSVM_2013_SpecTV,Gilboa_spectv_SIAM_2014}. In particular we discuss the representations introduced by a standard variational approach and by the Bregman iterations respectively inverse scale space methods (cf. \cite{rof92,digital_tv,iss}). The latter has been investigated in the case of total variation in \cite{Eckardt_Burger_Bach_thesis_2014}. Finally we shall also discuss the issue of the spectral response.

We start with a given image $f \in L^2(\Omega)$ and consider three different versions of the spectral transform.
The first one is the scale space from a gradient flow approach resembling the TV flow, the last one is obtained from the inverse scale space method. In between these methods there is the classical variational regularization.
A potentially confusing issue is the fact that time in the inverse scale space method rather corresponds to an inverse of the time variable in the other models. For this sake we will make change of time variables in the end to ease comparison. Moreover, in addition to the spectral representation $\phi$  we introduce another one called $\psi$ in the inverse time variable. Note that for small times $\phi$ measures changes in high frequencies and $\psi$ measures changes in low frequencies. In analogy to classical signal processing we call $\phi$ {\em wavelength representation} and $\psi$ {\em frequency representation}, respectively. Noticing that a change of variables $s=\frac{1}t$ yields
\begin{equation}
	\int_0^\infty \phi(t) w(t)~dt = \int_0^\infty \phi(\frac{1}s) w(\frac{1}s)\frac{1}{s^2}~ds,
\end{equation}
which motivates the consistency condition
\begin{equation}
	\psi(s) = \phi(\frac{1}s)\frac{1}{s^2} \quad \mbox{respectively} \quad \phi(t) = \psi(\frac{1}t)\frac{1}{t^2},
\end{equation}
that ensures the desirable inverse time relation between frequency and wavelength representation,
 $	\int_0^\infty \phi(t) w(t)~dt = \int_0^\infty  \psi(s) w(\frac{1}s) ~ds.$

For simplicity we assume $J(u) > 0 $ for $u \in X \setminus \{0\}$, which is usually achieved by choosing $X$ restricted in the right way (note that the null-space of a convex one-homogeneous functional is a linear subspace of $X$, \cite{Benning_Burger_2013}). E.g. in the case of total variation regularization we would consider the subspace of functions with vanishing mean value. The general case can be reconstructed by adding appropriate nullspace components.
The detailed definitions of the spectral representations are given as follows:
\begin{itemize}

\item[] {\bf Gradient Flow Representation: } Let $u_{GF}(t)$ be the solution of
\begin{equation}
\label{eq:gradientFlow}
	\partial_t u = - p_{GF}, \qquad p_{GF} \in \partial J(u)
\end{equation}
for $t > 0$ with initial value $u(0)=f$. Then the corresponding wavelength spectral transform is defined by
\begin{equation}
	\phi_{GF}(t) = t \partial_{tt} u_{GF}(t) = - t \partial_t p_{GF}(t).
\end{equation}
We obtain the frequency representation as
\begin{equation}
	\psi_{GF}(s) = \frac{1}{s^3} \partial_{tt} u_{GF}(\frac{1}s) = - \frac{1}{s^3} \partial_t p_{GF}(\frac{1}s) = \frac{1}{s} \partial_s q_{GF}(s),
\end{equation}
where $q_{GF}(s)=p_{GF}(\frac{1}s).$


\vspace*{6pt}

\item[] {\bf Variational Representation: } Let $u_{VM}(t)$ be the minimizer of
\begin{equation}
	\frac{1}2 \Vert u - f \Vert_2^2 + t J(u) \rightarrow \min_{u \in {\cal X}}, \ \ \text{i.e.} \ \ u_{VM}(t) = f - t p_{VM}(t), \ p_{VM} \in \partial J(u_{VM}) .
\end{equation}
Then the corresponding high frequency spectral representation is defined by
\begin{equation}
	\phi_{VM}(t) = t \partial_{tt} u_{VM}(t) = - \partial_t (t^2 \partial_t p_{VM}(t)).
\end{equation}
The low frequency representation can be derived from the equivalent form of the variational problem
\begin{equation}
	\frac{s}2 \Vert v - f \Vert_2^2 + J(v) \rightarrow \min_{v \in {\cal X}},
\end{equation}
with the relation $v_{VM}(s) = u_{VM}(\frac{1}s)$, which yields
\begin{equation}
	\psi_{VM}(s) = \frac{1}s \partial_{tt} u_{VM}(\frac{1}s) = s \partial_s (s^2 \partial_s v_{VM}(s)).
\end{equation}

\vspace*{6pt}

\item[] {\bf Inverse Scale Space Representation: }
Let $v_{IS}(s)$ be the solution of
\begin{equation}
	\partial_s q_{IS} = f- v, \qquad q_{IS} \in \partial J(v)
\end{equation}
for $s > 0$ with initial value $v(0)=0$. Then the corresponding low frequency spectral representation is defined by
\begin{equation}
	\psi_{IS}(s) = \partial_{s} v_{IS}(s) = - \partial_{ss} q_{IS}(s).
\end{equation}
With $u_{IS}(t) = v_{IS}(\frac{1}t)$ we obtain
\begin{equation}
\label{eq:issPhi}
	\phi_{IS}(s) = - t^2 \partial_{t} u_{IS}(t) .
\end{equation}

\end{itemize}

Note that due to low regularity of $J$ we expect $\phi_{*}$ to be a measure in time (here $*$ stands for either GF, VM, or IS). This is seen immediately from the canonical example of $f$ being a (nonlinear) eigenfunction of $J$:

\begin{theorem}
Let $f \in X$ satisfy (\ref{eq:ef_problem}) for some $\lambda > 0$ and $u=f$. Then, with the definitions made in \eqref{eq:gradientFlow} -- \eqref{eq:issPhi} we have
\begin{equation}
	\phi_*(t) = f \delta_{\frac{1}\lambda}(t), \qquad \psi_*(s) = f \delta_{\lambda}(s)\qquad \mbox{for~}*=\mbox{GF,VM,IS}.
\end{equation}
\end{theorem}
\begin{proof}
It is straightforward to check that $u_{GF}=u_{VM}=(1-\lambda t)_+ f$ and the a calculation of the second derivative in a distributional sense yields the corresponding $\phi$. In a similar way one can check in the case of the inverse scale space method that $u_{IS}(s) = 0$ for $s < \lambda$ and $u_{IS}(s)=f$ for $s > \lambda$ is the solution (with piecewise linear $p$), see e.g. \cite{Benning_Burger_2013}.
\end{proof}

The above result confirms our intuition about the spectral decompositions, indeed the eigenfunctions give a pure spectrum and the position of certain wavelengths  respectively frequencies in the spectral domain is proportional respectively inversely proportional to the eigenvalue.

As in (\ref{eq:tv_recon}) we can reconstruct the signal from the spectral response, noticing again that we have no nullspace components:
\begin{theorem}
Let $u_*$ be such that $u_*(0)=f$ and $u_*(t) \rightarrow 0$ sufficiently fast for $t \rightarrow \infty$, for $*=$GF,VM,IS.
Then
\begin{equation}
	f = \int_0^\infty \phi_*(t)~dt = \int_0^\infty \psi_*(s)~ds \qquad \mbox{for~}*=\mbox{GF,VM,IS}.
\end{equation}
\end{theorem}
\begin{proof}
Due to the consistency relation it suffices to verify the reconstruction formula only for $\phi$ or $\psi$.
For the appropriate choice the result follows simply by integration by parts, noticing $u_*(t=0) = f$ and all terms at infinity vanish due to the decay of $u$.
\end{proof}

For the sake of brevity we shall restrict ourselves to the gradient flow case in the following arguments.
An interesting property concerns the orthogonality of remaining signal and the spectral transform. In a classical Fourier series we have a decomposition into orthogonal components, so there is natural orthogonality between the spectral part at a certain frequency and the remaining signal at lower or higher frequency (the sum of orthogonal components with higher or lower indices). An analogous property holds for our spectral decomposition, i.e. we expect $u(t)$ to be orthogonal to $\phi(t)$. This can be seen from a formal computation of $\frac{d}{dt} (J(u(t)))$ in two ways. First of all we have
\begin{equation}
 \frac{d}{dt} (J(u(t))) = \langle p(t), \partial_t u(t) \rangle = - \Vert p(t) \Vert^2. \nonumber
 \end{equation}
On the other, since for one-homogeneous convex functionals $J(u(t))=\langle p(t), u(t) \rangle$ holds, we find
\begin{equation}
 \frac{d}{dt} (J(u(t))) = \frac{d}{dt} (\langle p(t), u(t) \rangle) = \langle \partial_t p(t), u(t) \rangle + \langle p(t), \partial_t u(t) \rangle = - \frac{1}t \langle \phi(t), u(t) \rangle - \Vert p(t) \Vert^2. \nonumber
\end{equation}
Hence, comparing the terms we obtain the orthogonality relation.

Finally let us discuss the spectral response $S$, which was defined before as the $L^1$-norm of $\phi$ in the TV case as mentioned above. This choice is somewhat arbitrary and in particular difficult to generalize to other functionals, so we need to derive a different version of the spectral response that can be expressed solely in terms of the functional $J$, the $L^2$-norm of $f$, and the method used to derive $\phi$. Let us start with the gradient flow, for which it is natural to investigate the energy dissipation, i.e., we compute time derivatives of $J$. By the chain rule we (formally) find
\begin{equation}
 \frac{d}{dt} J(u(t)) = \langle p(t) , \partial_t u(t) \rangle = - \Vert p(t) \Vert^2. \nonumber
\end{equation}
Moreover, in the case of $J$ differentiable, i.e. $p(t)=J'(u(t))$, we have
\begin{equation}
 \frac{d^2}{dt^2} J(u(t)) =  -2 \langle \partial_t p(t), p(t) \rangle
= - 2 \langle J''(u(t)) \partial_t u(t), p(t) \rangle = 2 \langle J''(u(t)) p(t), p(t) \rangle, \nonumber
\end{equation}
which is a nonnegative quantity, since $J$ is convex. For the nonsmooth case we might encounter even more interesting spectral responses this way, the nonexistence of a classical second derivative e.g. allows to have concentrated parts. This can be made precise again in the case of $f$ being an eigenfunction for the eigenvalue $\lambda$, i.e., (\ref{eq:ef_problem}) holds for $u=f$. Then we know that $p(t)=\lambda f$ for $t< \frac{1}\lambda$ and $p(t)=0$ for larger times. Computing $\Vert p(t) \Vert$ and its derivative we immediately find
\begin{equation}
	\frac{d^2}{dt^2} J(u(t))  = \lambda^2 \Vert f \Vert^2 \delta_{\frac{1}\lambda}(t) = \frac{1}{t^2} \Vert f \Vert^2 \delta_{\frac{1}\lambda}(t) .
\end{equation}
Since apparently $\Vert f \Vert$ is a suitable value for the magnitude, we define the spectral response as
\begin{equation} \label{eq:newSdefinition}
	S(t) = t \sqrt{ \frac{d^2}{dt^2}J(u(t)) } = \sqrt{\langle \phi(t), 2 t p(t)} \rangle.
\end{equation}
With this definition we have the following analogue of the Parseval identity, which follows again using integration by parts and sufficient decay of $u$ and its derivatives:
\begin{eqnarray}
	\Vert f \Vert^2 & = & - \int_0^\infty \frac{d}{dt} \Vert u(t) \Vert^2 ~dt = 2 \int_0^\infty \langle p(t), u(t) \rangle ~dt =
	 2 \int_0^\infty J(u(t))~dt \nonumber \\ & = & \int_0^\infty S(t)^2  ~dt,	\end{eqnarray}
which confirms that the spectral representation encodes the full norm of $f$.

\vspace*{-12pt}

\begin{figure}[htb]
\begin{center}
\begin{tabular}{ ccc  }
\includegraphics[width=0.2\textwidth]{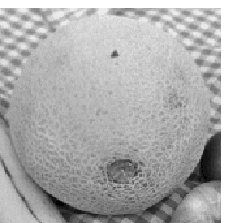}&
\includegraphics[width=0.3\textwidth]{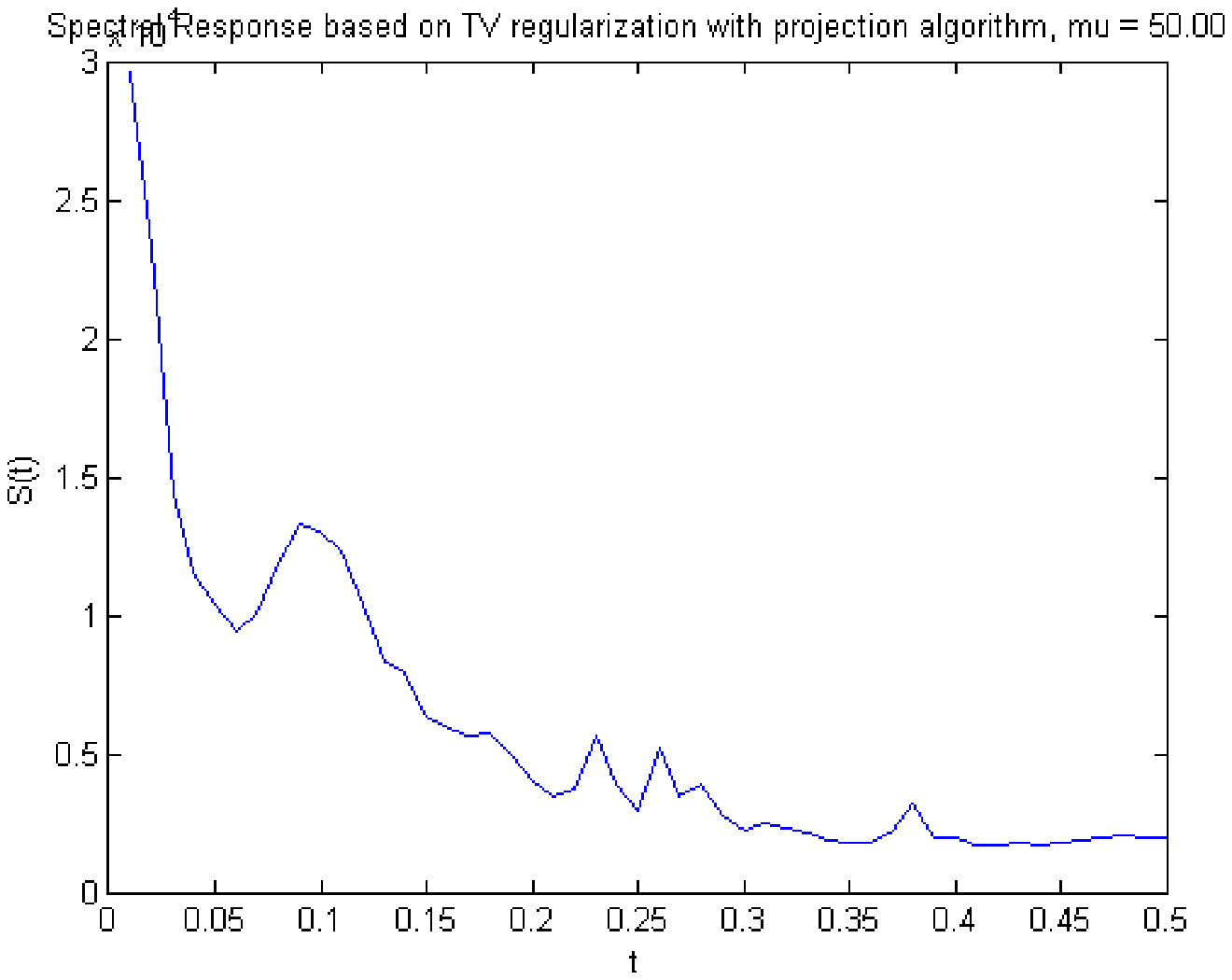}&
\includegraphics[width=0.3\textwidth]{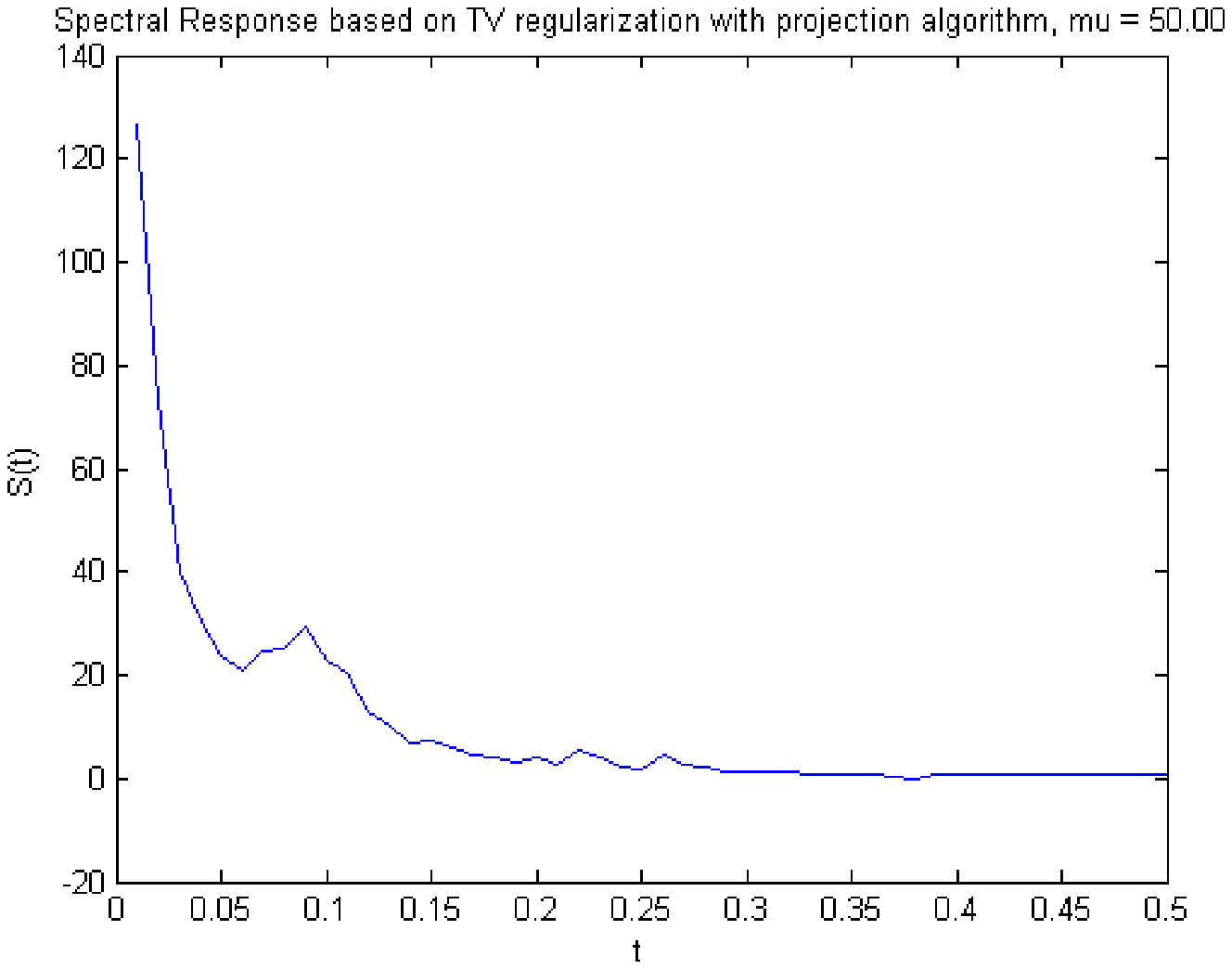} \\
Melon image & $\Vert \phi(t) \Vert_{L^1}$ & $   t^2 \frac{d^2}{dt^2}J(u(t)) $
\end{tabular}
\caption{Spectral resonse for the total variation flow in an image of a melon (left). The original definition from \cite{Gilboa_SSVM_2013_SpecTV} in the middle and the definition (\ref{eq:newSdefinition}) on the right.}
\label{fig:spectralresponse}
\end{center}
\vspace{-0.5cm}
\end{figure}

\vspace*{-12pt}
In computational experiments, the behaviour $S$ as defined by (\ref{eq:newSdefinition}) appears to be more suitable for the spectral decomposition, in particular there are the important maxima marking essential changes in the spectrum but less further oscillations. This is illustrated for total variation flow on an image already considered in \cite{Gilboa_SSVM_2013_SpecTV} in Figure \ref{fig:spectralresponse}. We finally mentioned that a generalization of this definition to variational methods and in particular inverse scale space methods is a nontrivial task and beyond the scope of this paper.

\section{Examples}

In the following we discuss three interesting examples of one-homogeneous functionals beyond total variation, which all provide interesting spectral definitions in a different context.
\subsection{Recovering Fourier Analysis}
 One particular example for a one-homogeneous regularization functional $J$ which restores a very classical method is $J(u) = \|Vu\|_1$ for an orthonormal linear transform $V$, e.g. for $V$ corresponding to frequency analyzing transforms like the Fourier or cosine transforms. A similar computation has previously been done by Xu and Osher in \cite{XuOsher2007} in which they analyzed Bregman iteration as well as the inverse scale space flow for $V$ corresponding to a wavelet transformation.

For the sake of simplicity let us consider the discrete problem. For any orthonormal $V$ we obtain
\begin{eqnarray}
\min_u ~ \frac{1}{2} \|u-f\|_2^2 + t \|Vu\|_1, &=& \min_z ~\frac{1}{2}\|V^Tz-f\|_2^2 + t \|z\|_1 ,\nonumber\\ 
&=& \min_z ~ \frac{1}{2}\|z-Vf\|_2^2 + t \|z\|_1 ,\nonumber
\end{eqnarray}
where we have used the orthonormality of $V$ along with the fact that the $\ell^2$ norm is invariant with respect to multiplication with orthonormal matrices. Note that $Vf$ becomes the frequency transform of the input data to which we apply $\ell^1$ regularization. The above minimization problem in $z$ admits a closed form solution known as soft-shrinkage, i.e.
\begin{equation}
z_{VM}(t) = \arg \min_z  \frac{1}{2}\|z-Vf\|_2^2 + t \|z\|_1  = \textnormal{sign}(Vf) \; \max (|Vf|-t,0),
\end{equation}
and hence $u_{VM}(t) = V^Tz_{VM}(t)$. The first time derivative of $z_{VM}(t)$ is
\begin{equation}
\partial_t \left(z_{VM}(t)\right)_i = \left \{ \begin{array}{cc} 0 & \textnormal{ if } |(Vf)_i|<t, \\ -\textnormal{sign}((Vf)_i)& \textnormal{ else. } \end{array} \right. .
\end{equation}
Interestingly, we can see that this means that $\partial_t z_{VM}(t)$ is in its own negative subdifferential, i.e.
\begin{equation}
\partial_t z_{VM}(t) \in - \partial \|z_{VM}\|_1,
\end{equation}
from which we can conclude that $z_{VM}(t)=z_{GF}(t)$ and $u_{VM}(t)=u_{GF}(t)$. The second time derivative of $z_{*}(t)$ (for $* \in \{VM, GF \}$) becomes
\begin{equation}
\partial_{tt} \left(z_*(t)\right)_i =  \textnormal{sign}((Vf)_i) ~ \delta_{t = |(Vf)_i|}. \nonumber
\end{equation}
The wavelength transform function $\phi^z_{*}(t) = t \partial_{tt} z_{*}$ is therefore given as
\begin{equation}
\left(\phi^z_*(t)\right)_i= \delta_{t = |(Vf)_i|}~ (Vf)_i. \nonumber
\end{equation}
As we can see the spectral transform $\phi^z_*$ simply reduces to the spectrum of the frequency coefficients $Vf$ in this case. Moreover, the first possible definition of the spectrum,
\begin{equation}
 S^z_*(t) = \|\phi^z_*(t)\|_1 = |\{i ~|~ |(Vf)_i|=t\}| ~ t, \nonumber
 \end{equation}
reduces to the sum over the absolute values of the frequencies that have magnitude $t$, when considered with respect to the $z$ variable. With respect to the variable $u$ we obtain
 $S^u_*(t) = t \|V^T e_{|(Vf)_i|=t}\|_1,$
where $ e_{|(Vf)_i|=t}$ denotes the vector with the $i$th entry being one if $|(Vf)_i|=t$ and zero else. Since for random DCT coefficients of $f$ the vector $e_{|(Vf)_i|=t}$ is at most one-sparse with probability one, the peaks of the spectrum have magnitude $|(Vf)_i| \|v_i\|_1$ where $v_i$ is the $i$th row of $V$. Since $\|v_i\|_1$ is not the same for all $i$ and also does not have a specific interpretation, definition \eqref{eq:newSdefinition}, i.e.
\begin{equation}
	S(t) = t \sqrt{ \frac{d^2}{dt^2}J(u(t))} = t \sqrt{ \frac{d^2}{dt^2} \|Vu(t)\|_1 } = t \sqrt{ \frac{d^2}{dt^2} \|z(t)\|_1}
\end{equation}
is preferable since it does not depend on the $u$ or $z$ representation. Since the first time derivative of $ \|z(t)\|_1 $ is just the number of nonzero entries of $z$, we obtain $S(t) = \sqrt{|\{i ~|~ |(Vf)_i|=t\}|} ~ t$.

Finally, let us consider the third possible definition of a spectral definition, i.e. the inverse scale space flow. By following \cite{XuOsher2007}, we obtain
\begin{equation}
z_{IS}(s) = \left \{ \begin{array}{cc} 0 & \textnormal{ if } s|(Vf)_i|< 1, \\ -\textnormal{sign}((Vf)_i)& \textnormal{ else. } \end{array} \right. ,
\end{equation}
such that for $t = 1/s$ we obtain $z_{IS}(t)=z_{VM}(t)=z_{GF}(t)$ and all three definitions of wavelength and frequency representations coincide.

Figure \ref{fig:dctExample} shows an example of the above setup using the discrete cosine transform (DCT) for $V$. As we can see, the low pass reconstructions are obtained as hard thresholdings of the coefficients, which is what we refer to as an ideal low pass filter.

\vspace*{-12pt}
\begin{figure}[bth]
\begin{tabular}{ ccc }
\includegraphics[width=0.325\textwidth]{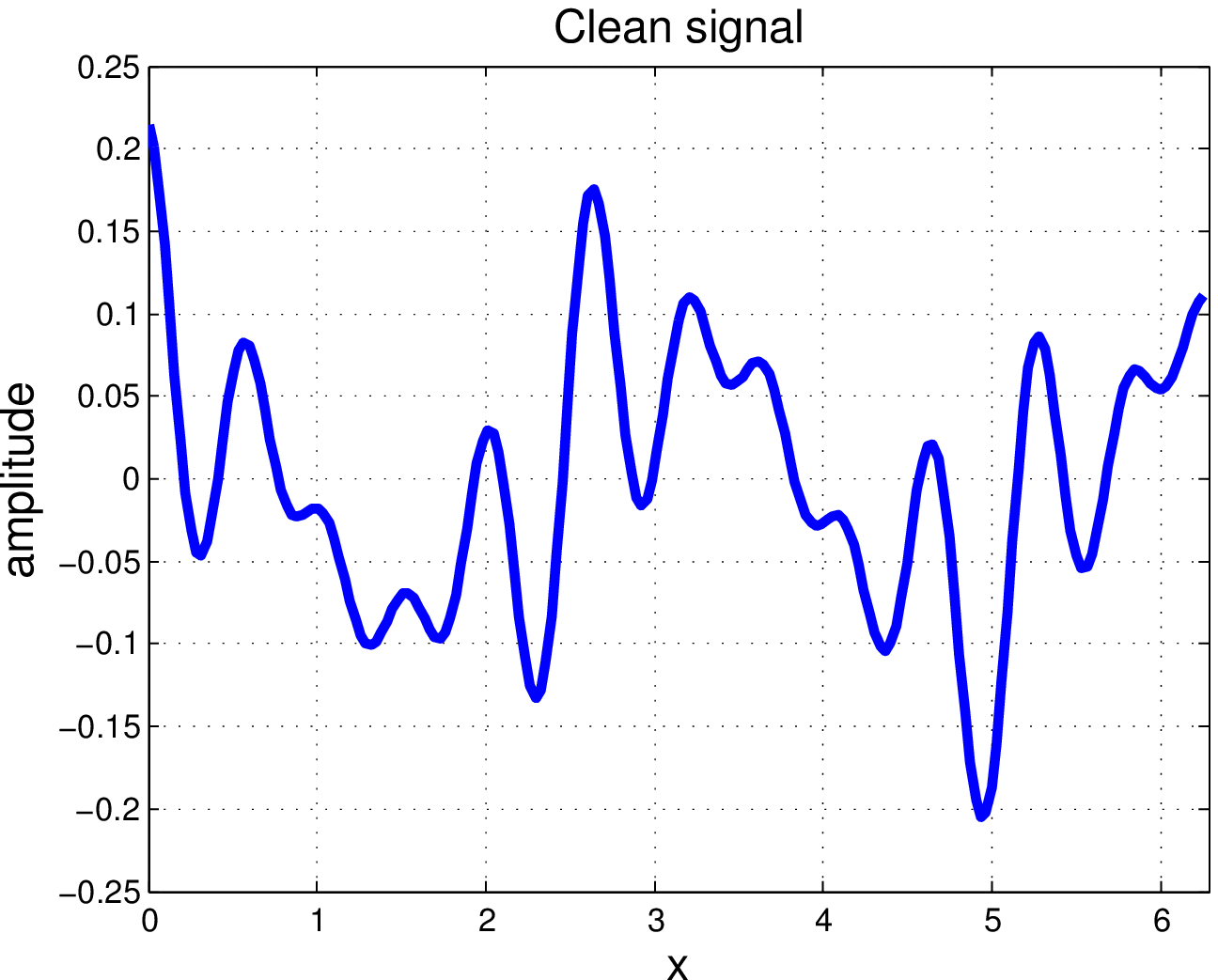}
\includegraphics[width=0.325\textwidth]{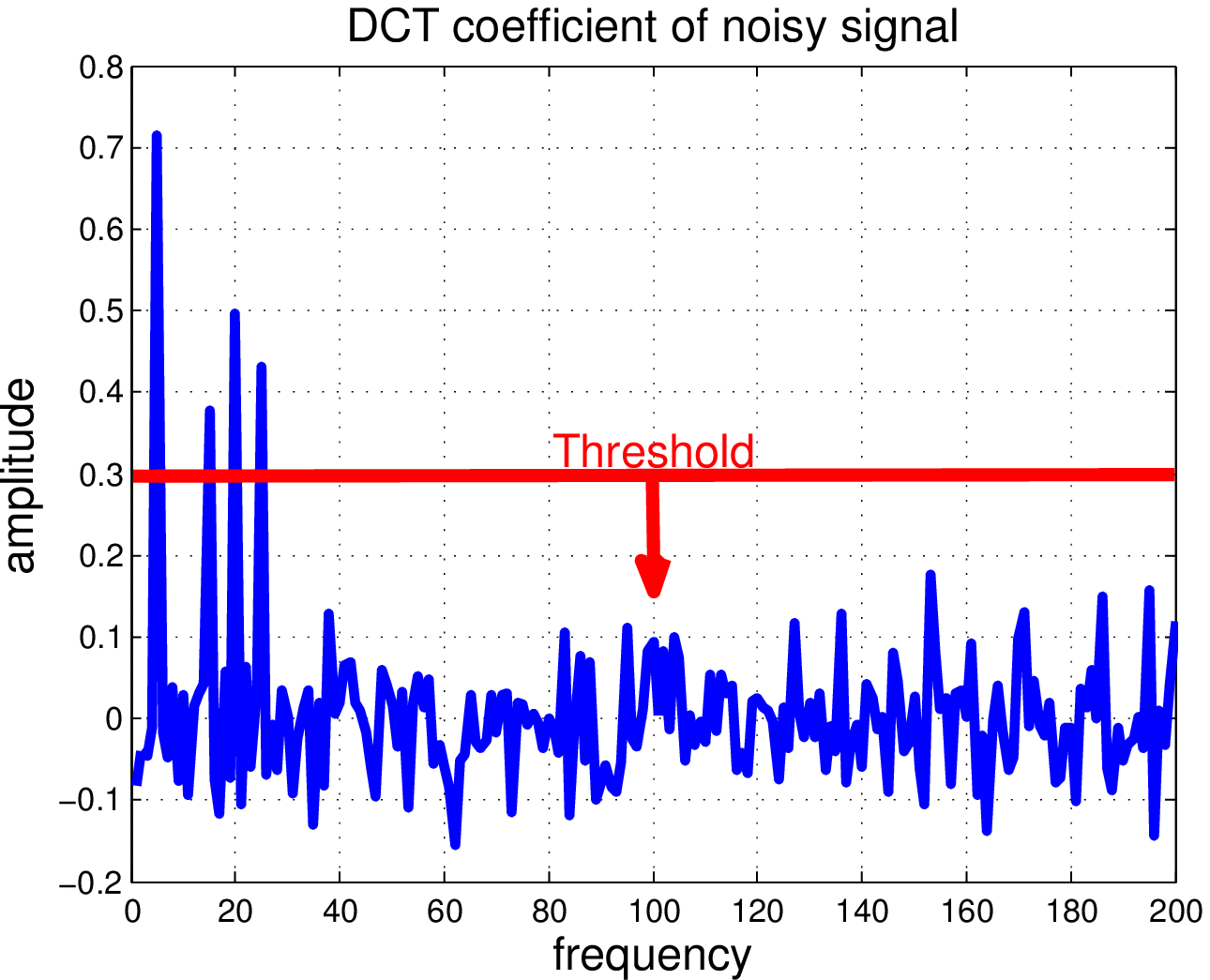}
\includegraphics[width=0.325\textwidth]{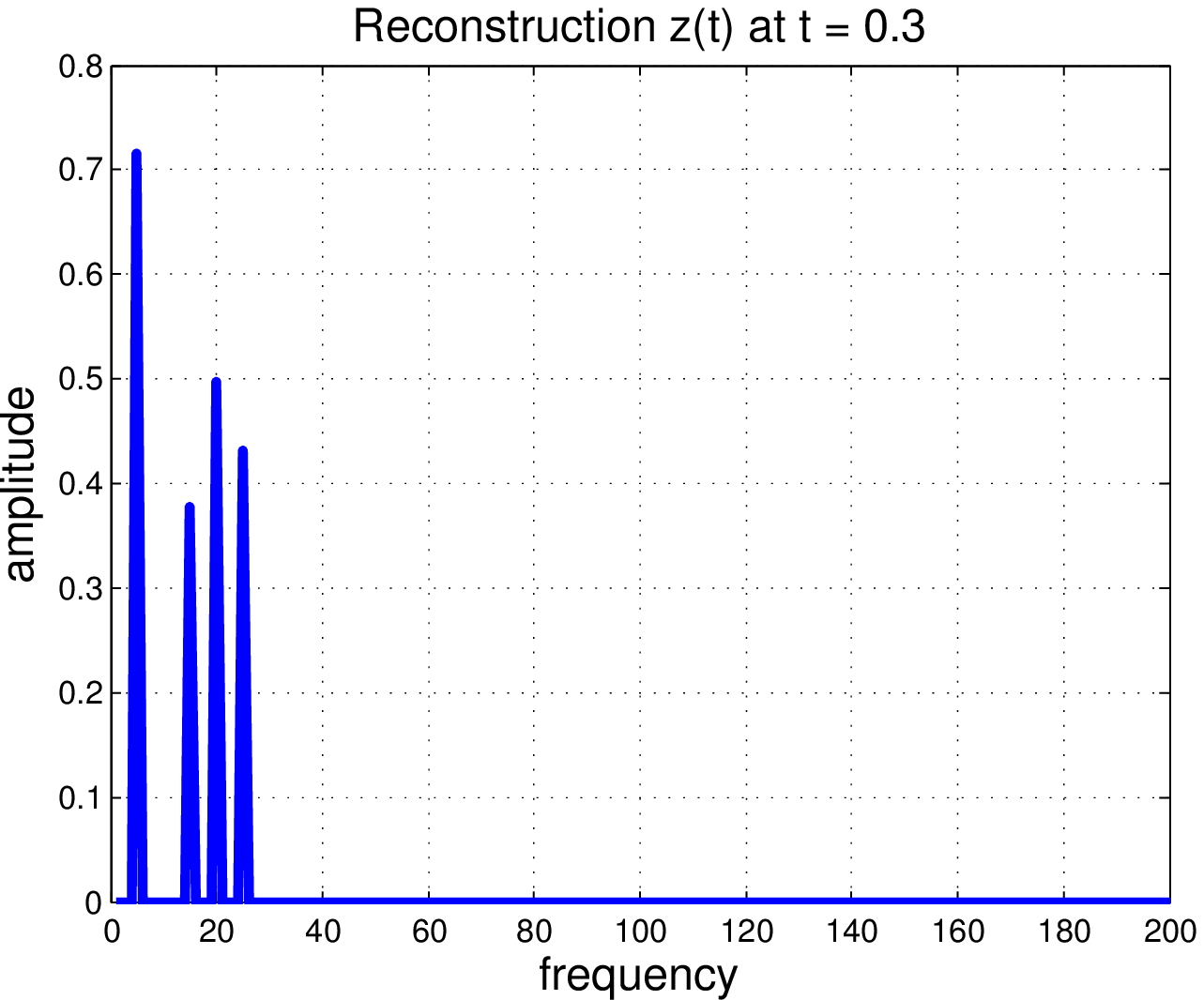}\\
\includegraphics[width=0.325\textwidth]{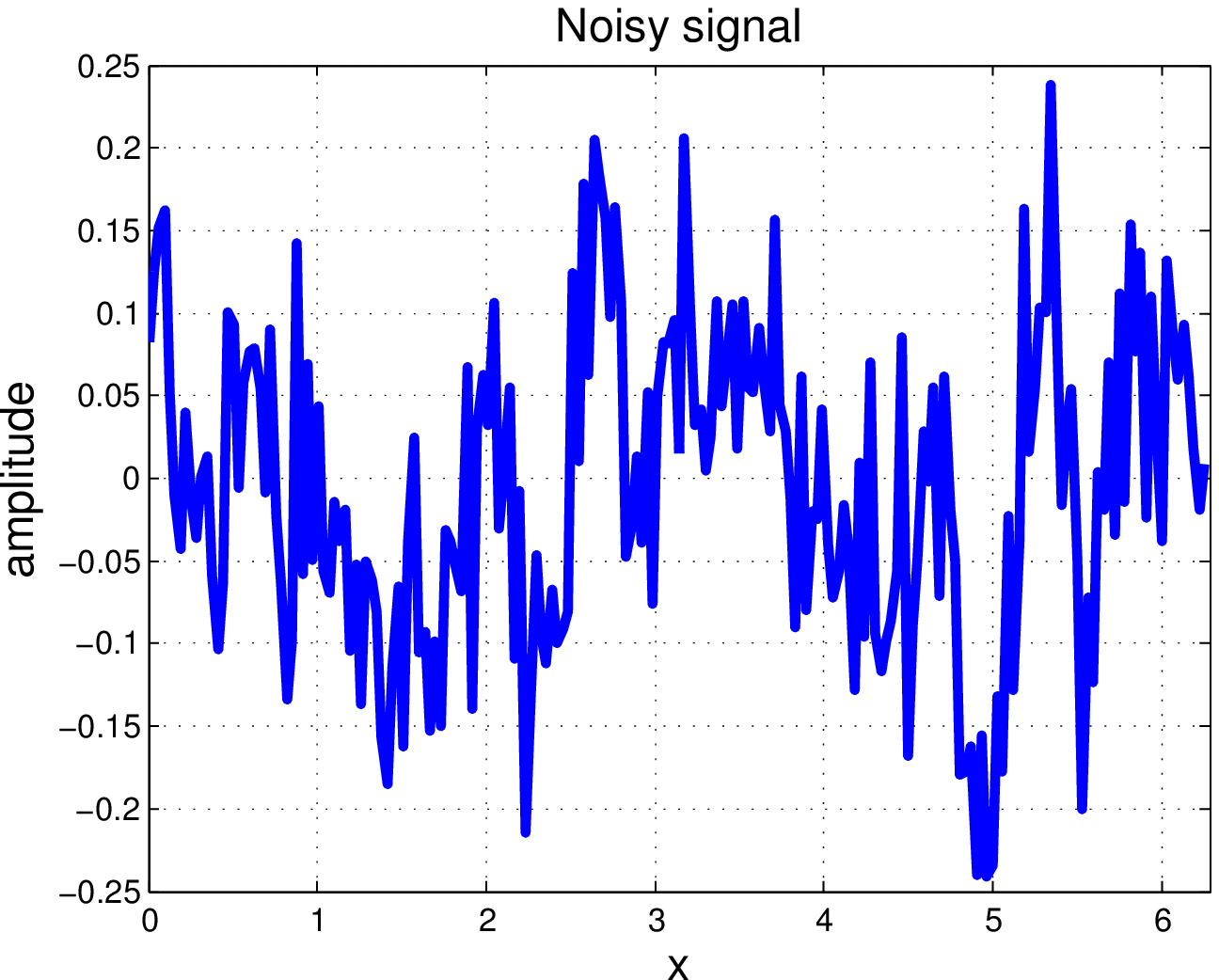}
\includegraphics[width=0.325\textwidth]{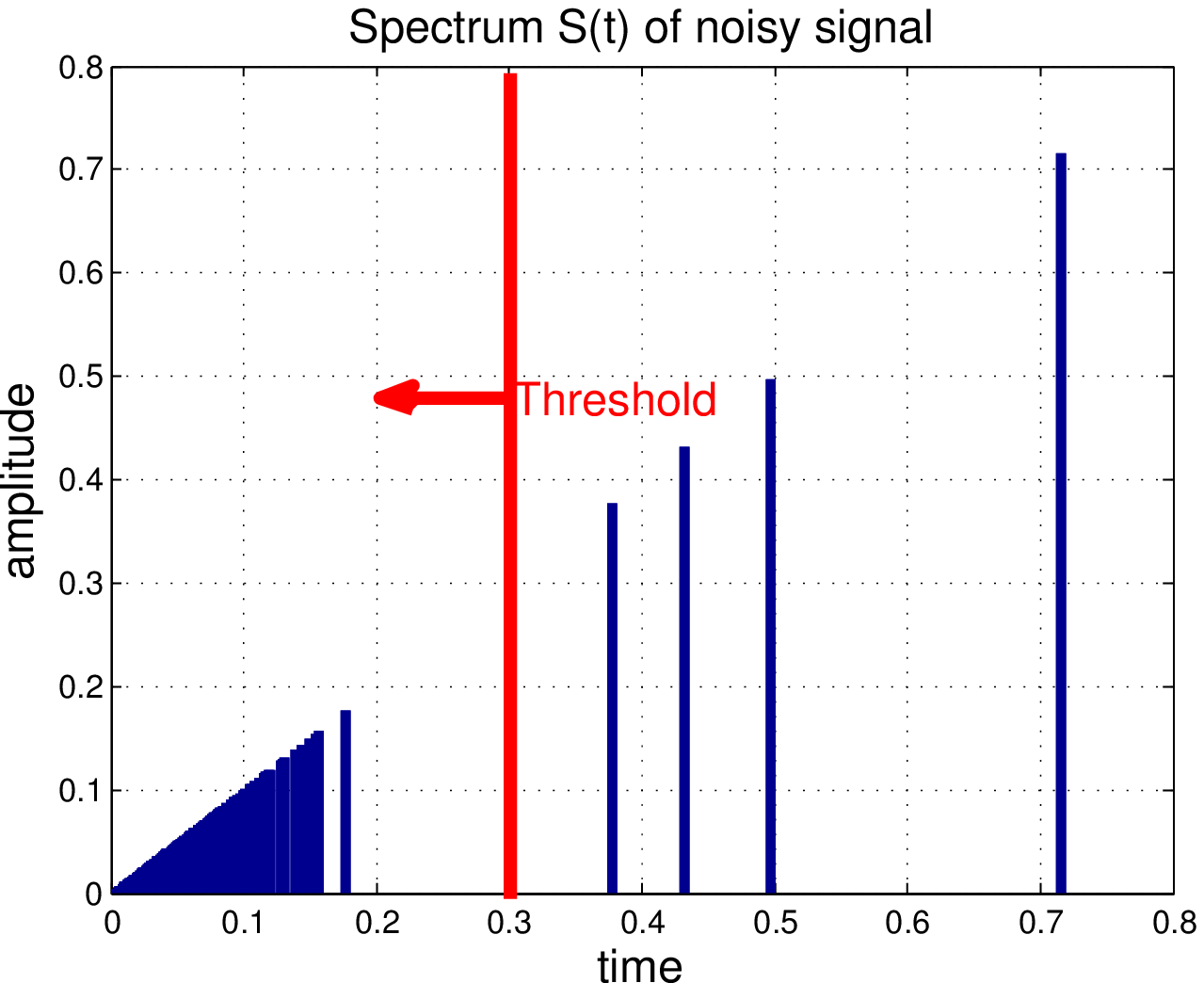}
\includegraphics[width=0.325\textwidth]{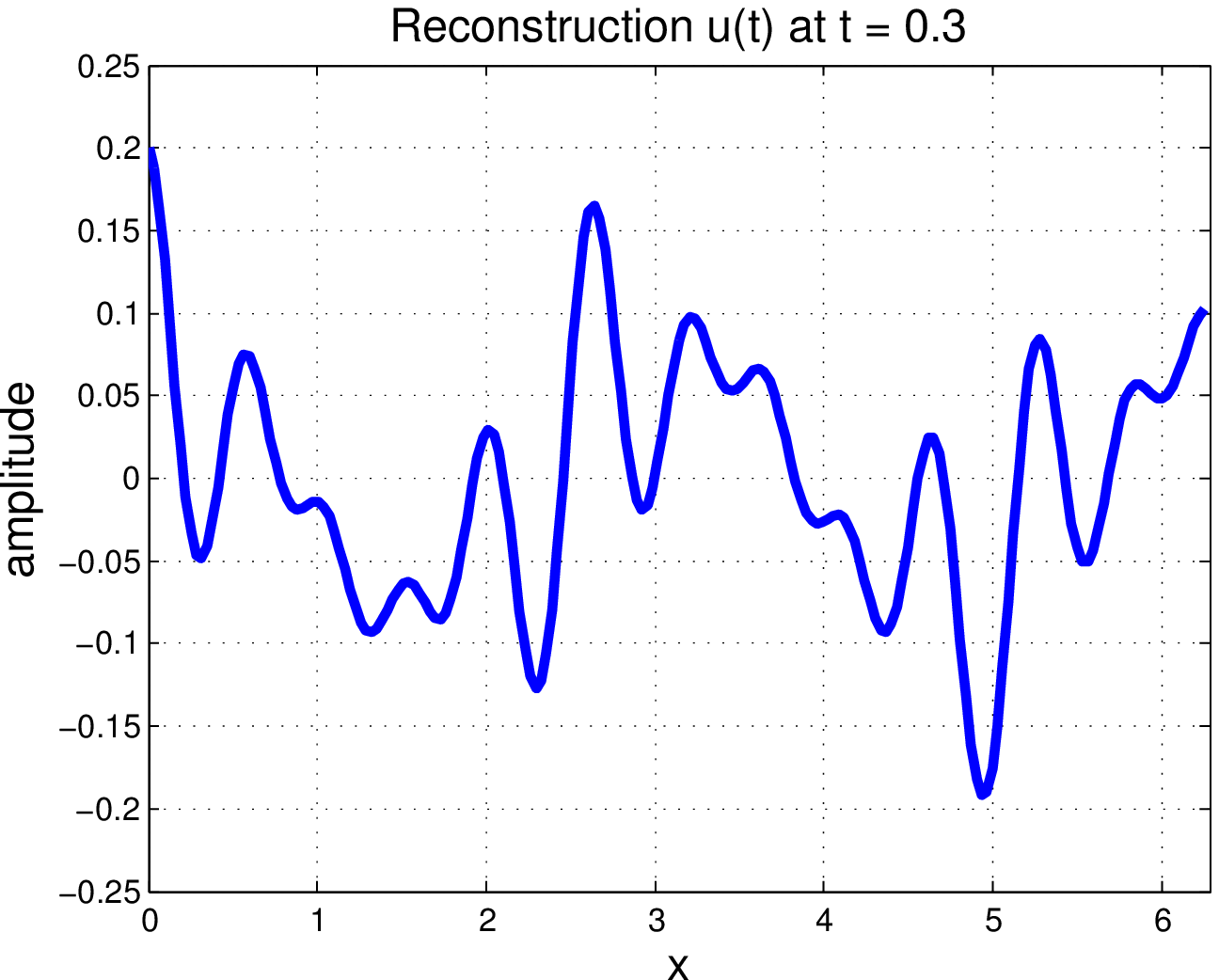}
\end{tabular}
\caption{Recovering classical spectral analysis: The upper left plot shows a clean signal suitable for classical spectral analysis, and below is its noisy version. The upper middle image shows the classical processing/denoising of such data. One determines the DCT coefficients and eliminates contributions of frequencies whose magnitude is below a certain threshold. The lower middle image shows the spectrum $S(t)$ we obtain with the proposed approach with $J(u) = \|Vu\|_1$ for all three variants of the spectral analysis. Due to the obvious separation of four of the peaks from the remaining ones, one could similarly choose to reconstruct the solution at time $t = 0.3$, which is equivalent to the threshold of $0.3$ in the classical sense. The right plots show $z(0.3)$ and $u(0.3)$, i.e. the reconstruction of the ideal filter in the DCT and spatial domains. }
\label{fig:dctExample}
\end{figure}

\vspace*{-24pt}

\subsection{Total Generalized Variation}
Let us now look at a highly nonlinear variational approach and consider the second order total generalized variation (TGV) of \cite{bredies_tgv_2010}, which -- for sufficiently regular $u$ -- can be written as
$ J(u) = \min_{\nabla u = v +z } \beta \|v\|_1 + (1-\beta) \|\nabla z\|_1. $%

Figure \ref{fig:tgvExample} shows an analysis similar to the classical one of Figure \ref{fig:dctExample}, but now for the inverse TGV flow using $\beta = 0.05$. We provide the full movie showing the $u(t)$ throughout an evolution of 1500 time steps in the supplementary material of this paper.
 As we can see, applying an ideal low pass filter leads to almost perfect reconstructions for signals consisting of piecewise linear parts and jumps.

This general behavior can be understood by considering eigenfunctions of the TGV. Examining our previous observations on the favorable properties of the spectral decomposition techniques on eigenfuctions  as well as the observations in \cite{Benning_Burger_2013}, eigenfunctions should reveal the types of signals the regularization 'likes' in the sense that they can be reconstructed easily. Recalling the TGV eigenfunction analysis in \cite{Muller_thesis,Benning_highOrderTV_2013} one can observe that indeed the appearance of the eigenfunction coincides with our numerical observations in the spectral analysis: The original signal $f$ is approximated by sequences of piecewise linear functions with discontinuities.

\begin{figure}[h]
\begin{center}
\includegraphics[width=0.325\textwidth]{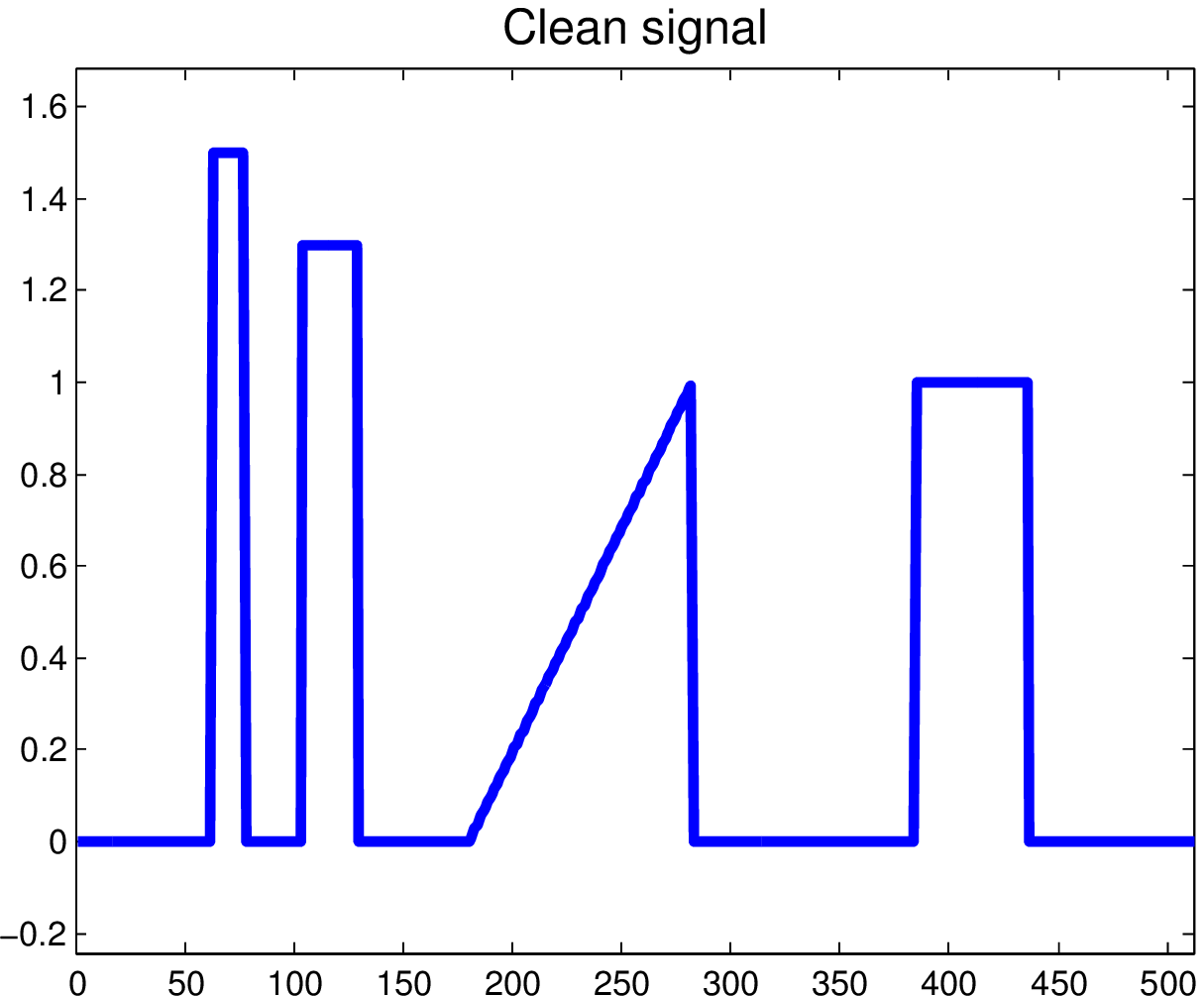}
\includegraphics[width=0.325\textwidth]{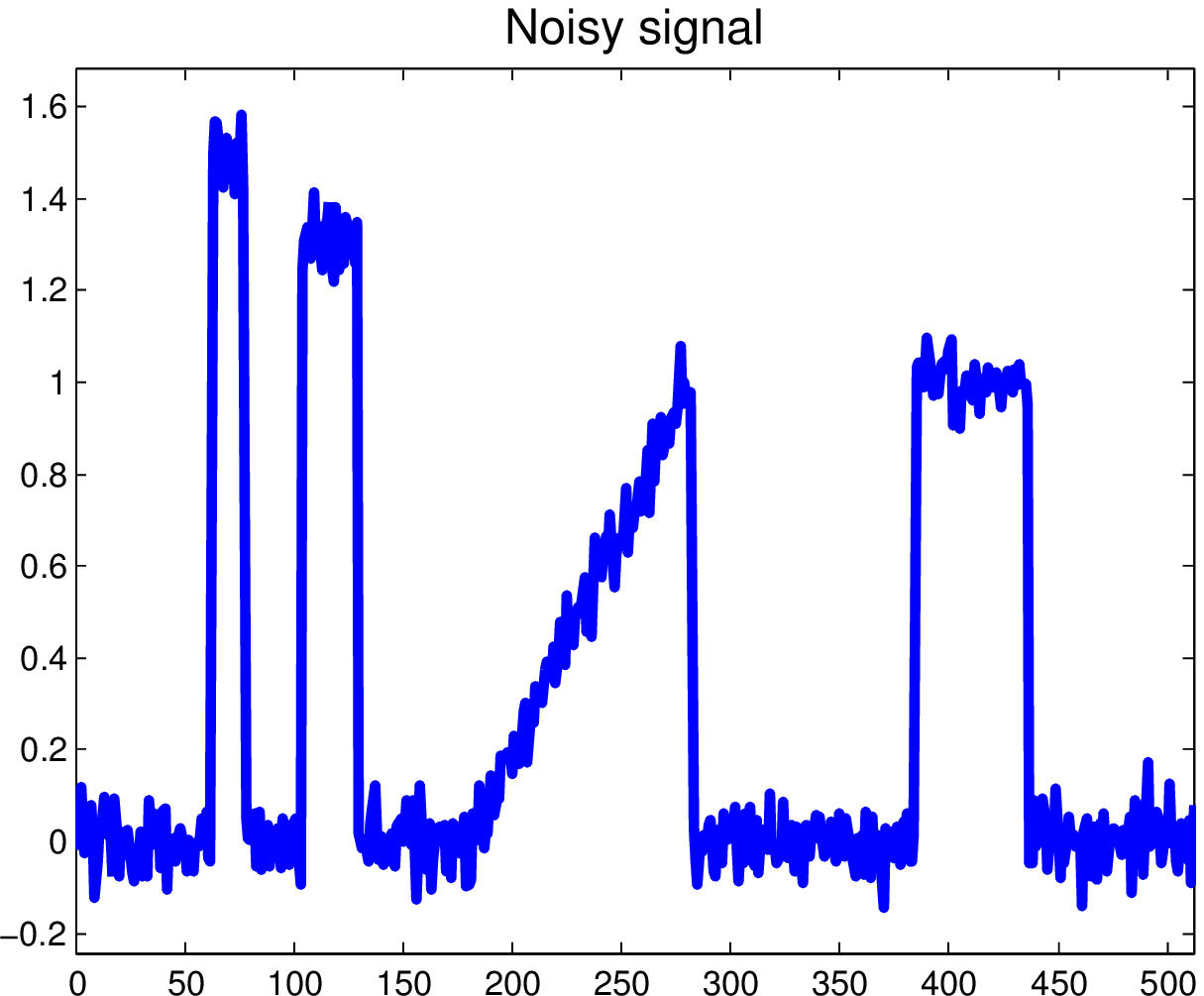}
\includegraphics[width=0.325\textwidth]{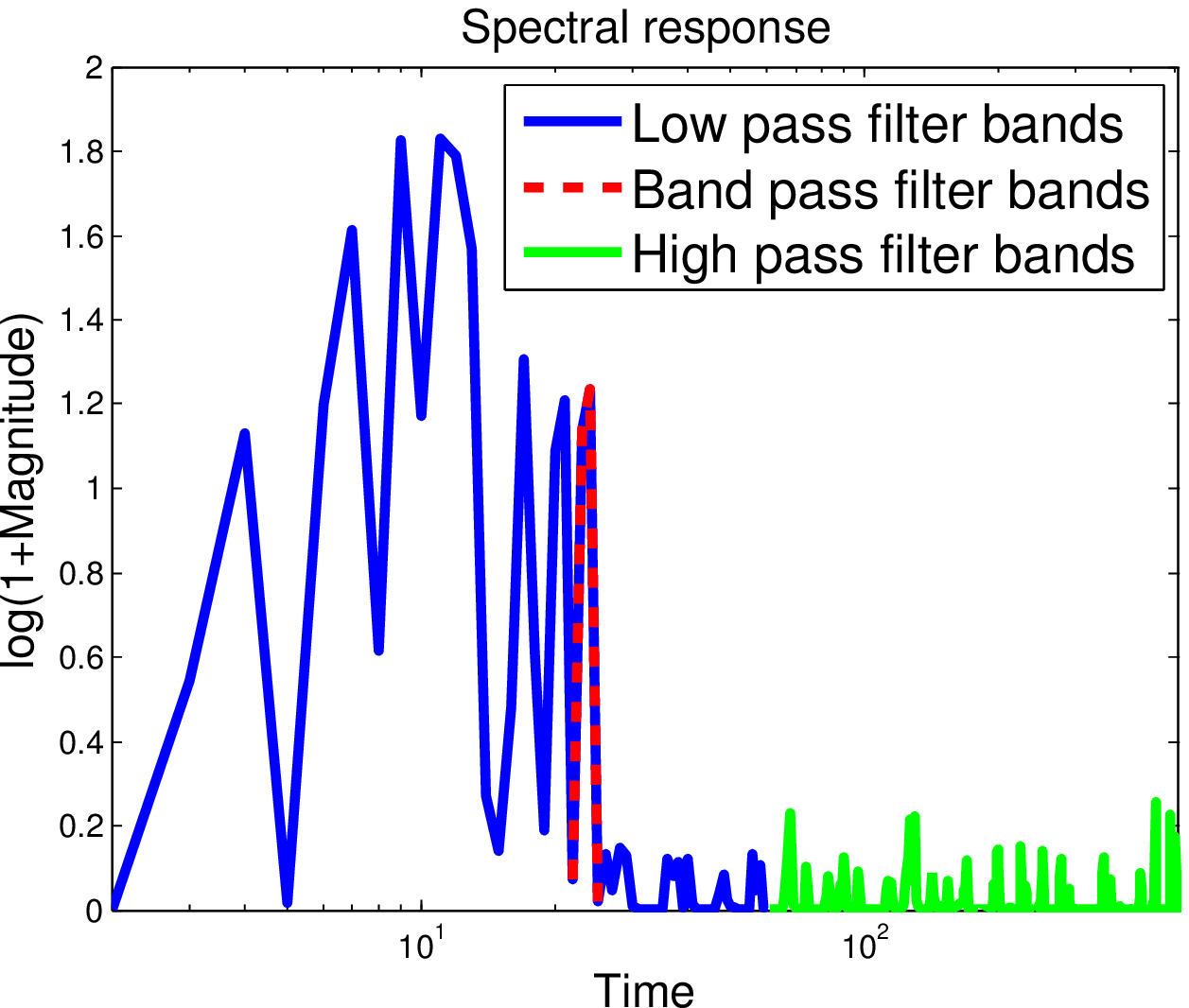} \\
\includegraphics[width=0.325\textwidth]{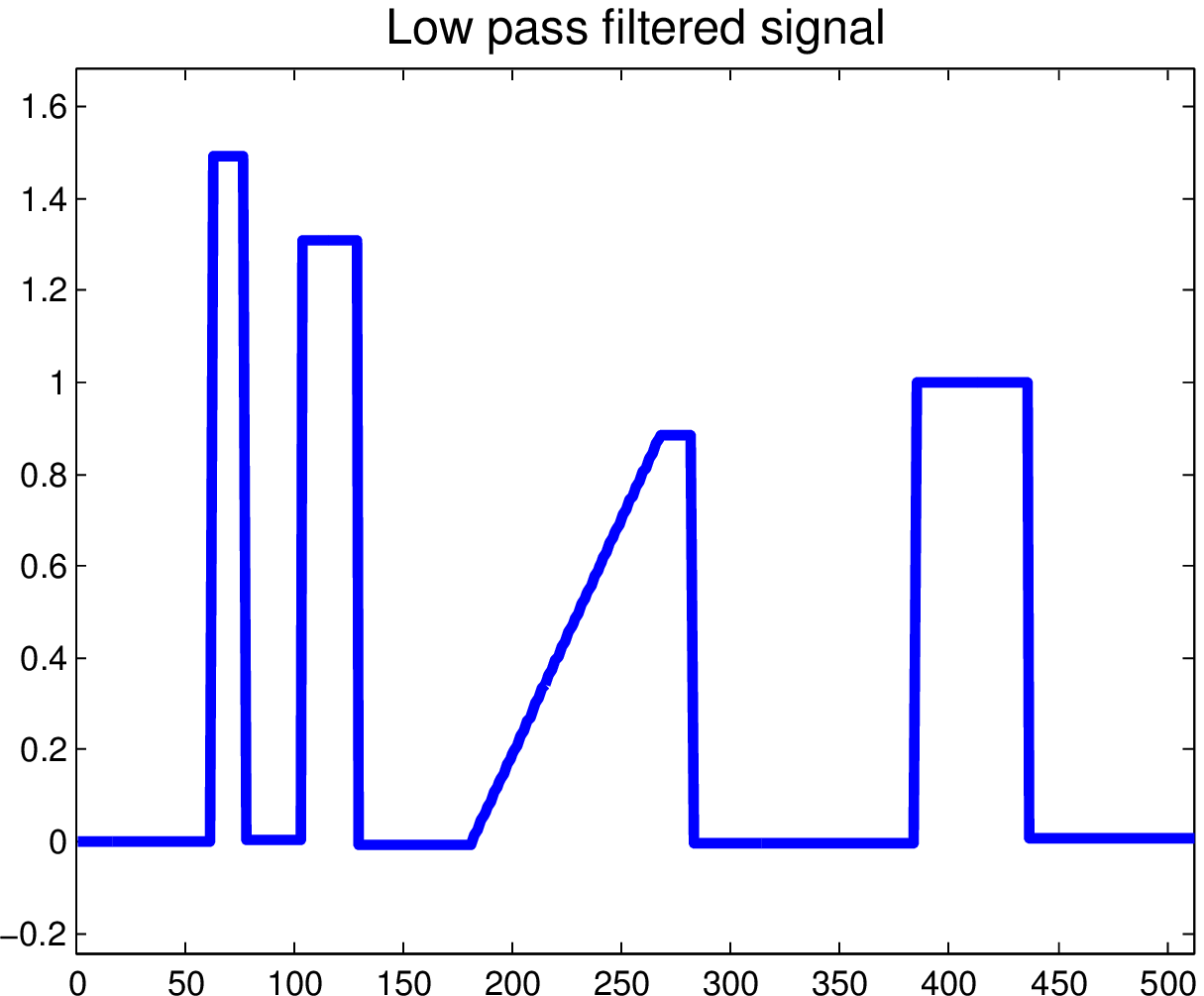}
\includegraphics[width=0.325\textwidth]{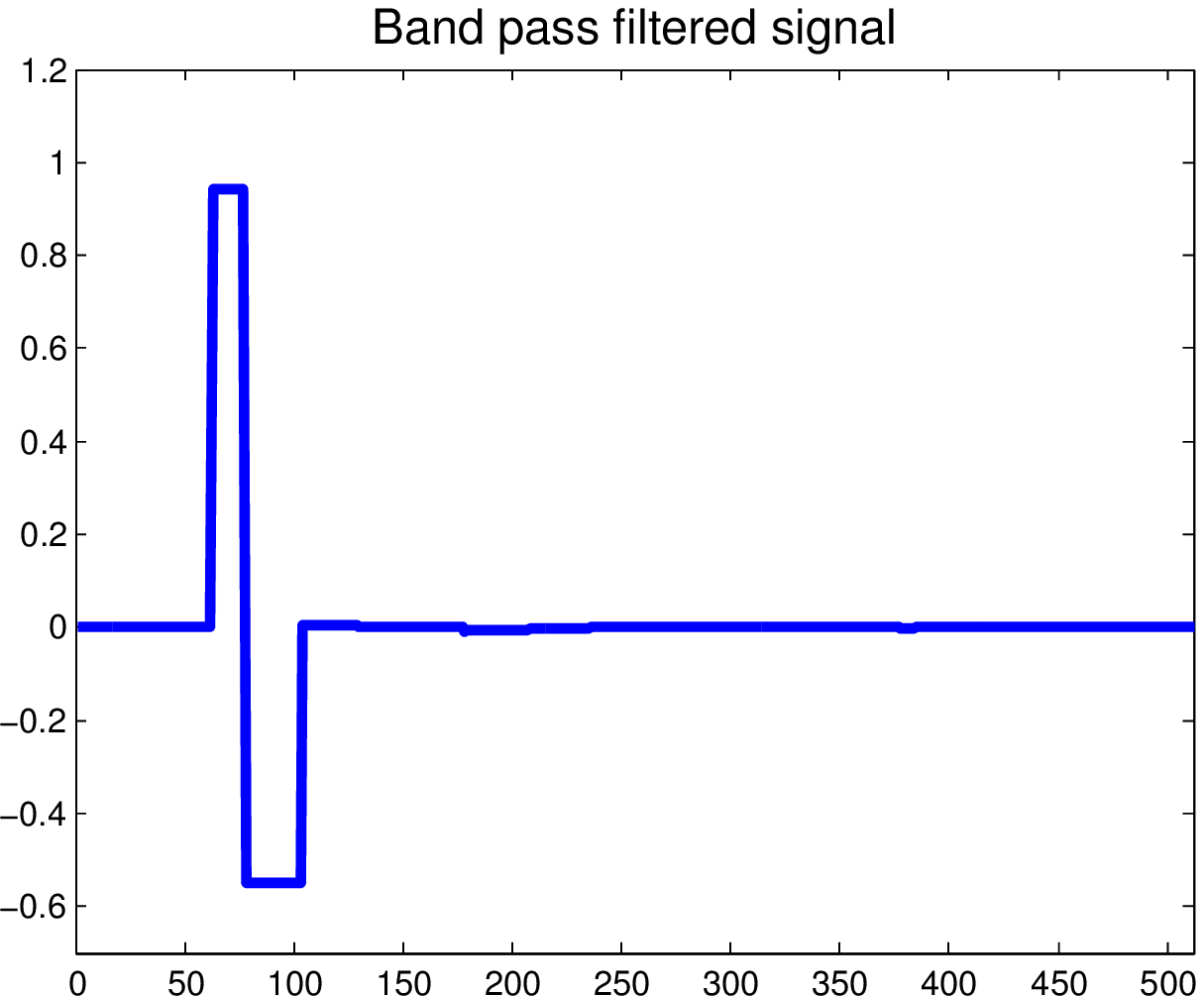}
\includegraphics[width=0.325\textwidth]{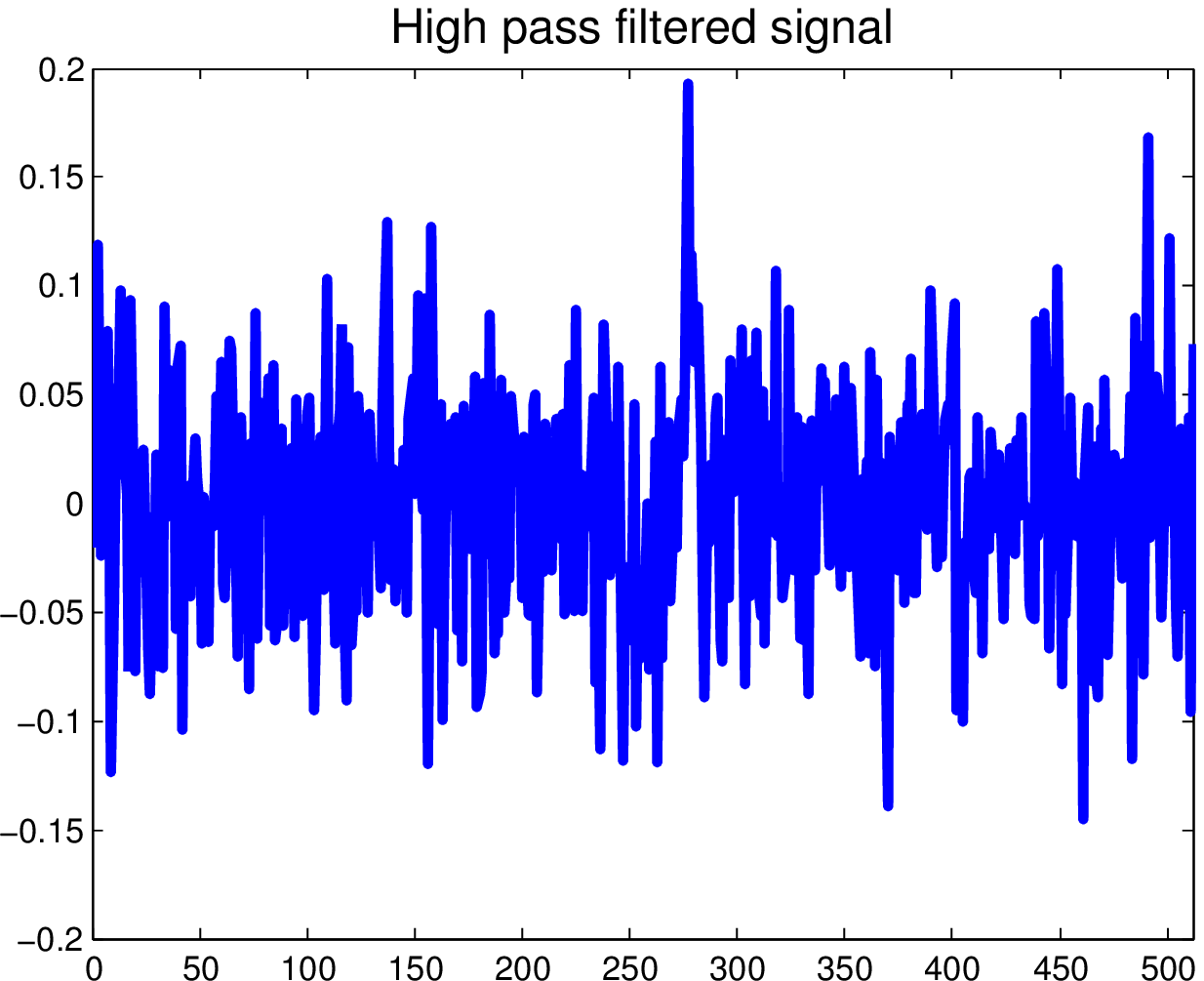}
\caption{Inverse scale space spectral analysis of a signal using TGV regularization. After we added noise to the clean signal (upper left) we obtain the noisy version shown in the upper middle. The spectrum we obtain using TGV analysis with $S(t) = \|\phi(t)\|_1$ is shown in the upper right in a semi logarithmic plot. By low pass, band pass or high pass filtering the spectral decomposition (marked blue, red and green in the spectrum) we obtain the lower left, middle, and right signals respectively. As we can see one can obtain nice reconstructions of the true signal by the low pass filtering. Certain frequencies (like the two peaks of the true signal) are contained in intermediate bands and can be isolated by band pass filtering. The high pass filter contains mostly noise.}
\label{fig:tgvExample}
\end{center}
\end{figure}

\vspace*{-30pt}

\subsection{Coupled Signal Analysis}
As a third example, we'd like to mention that spectral decomposition could also serve as a tool for collaboratively analyzing input data. Just to point out possible applications, Figure \ref{fig:collaboartaionExample} shows two examples: The left part corresponds to an analysis of 15 different input signals $f \in \mathbb{R}^{n \times 15}$ with the collaborative sparse regularization $J(u)=\|u\|_{\infty, 1} = \sum_i \max_j |u_{i,j}|$. Similar to the DCT case, the problem decouples in one direction and yields $\ell^\infty$ regularized subproblems in each component. As we can see in the left part of figure \ref{fig:collaboartaionExample}, an ideal low-pass reconstruction yields exactly those peaks for which all input signals in $f$ had non-zero entries.

Similar concepts could also be derived for collaborative jumps as shown in the right of \ref{fig:collaboartaionExample}, in which we used $J(u) = \|\nabla u\|_{\infty ,1}$. As we can see, a strong low-pass filter leads to the best single joint jump approximation in our example.

\begin{figure}[htb]
\includegraphics[width=0.245\textwidth]{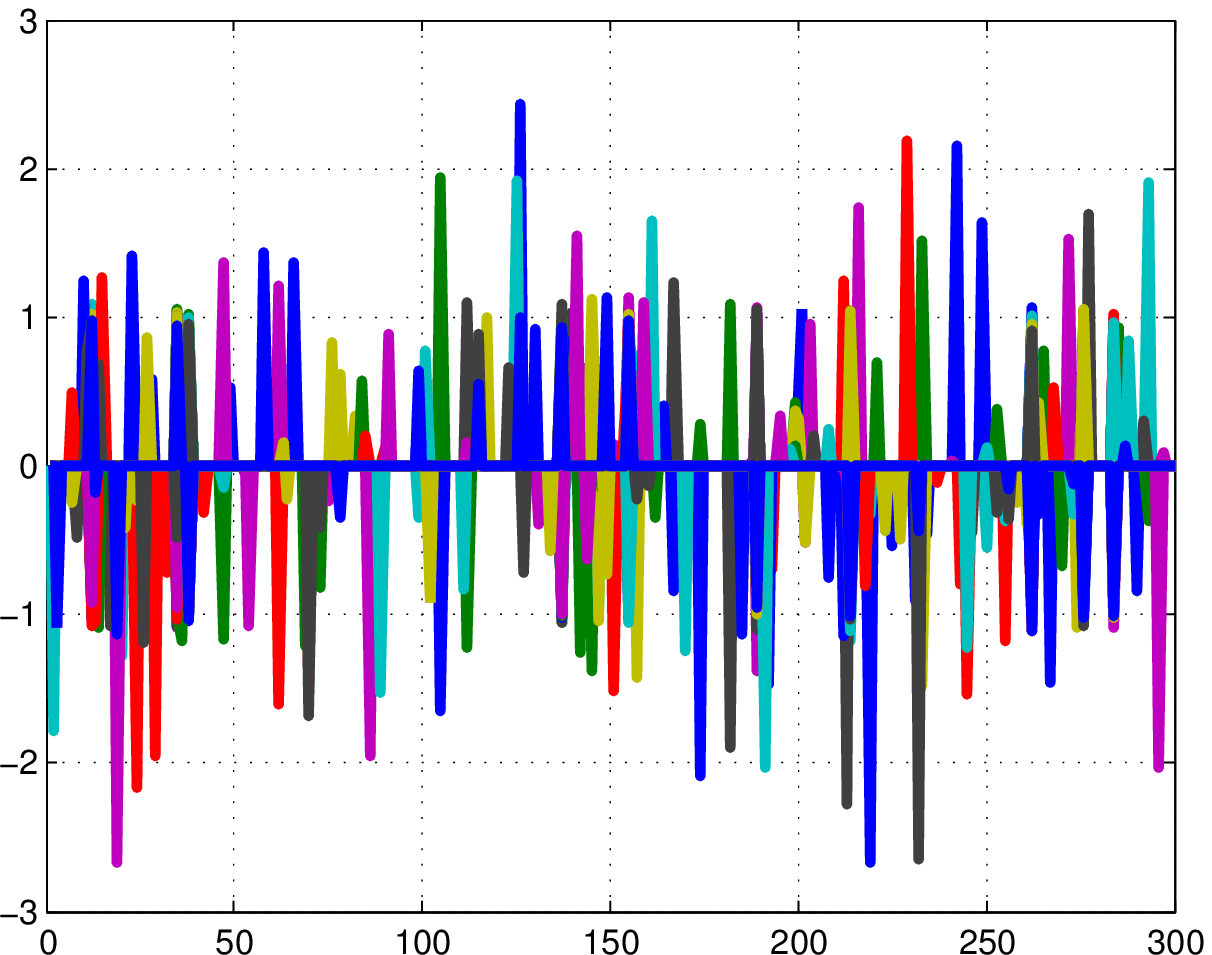}
\includegraphics[width=0.245\textwidth]{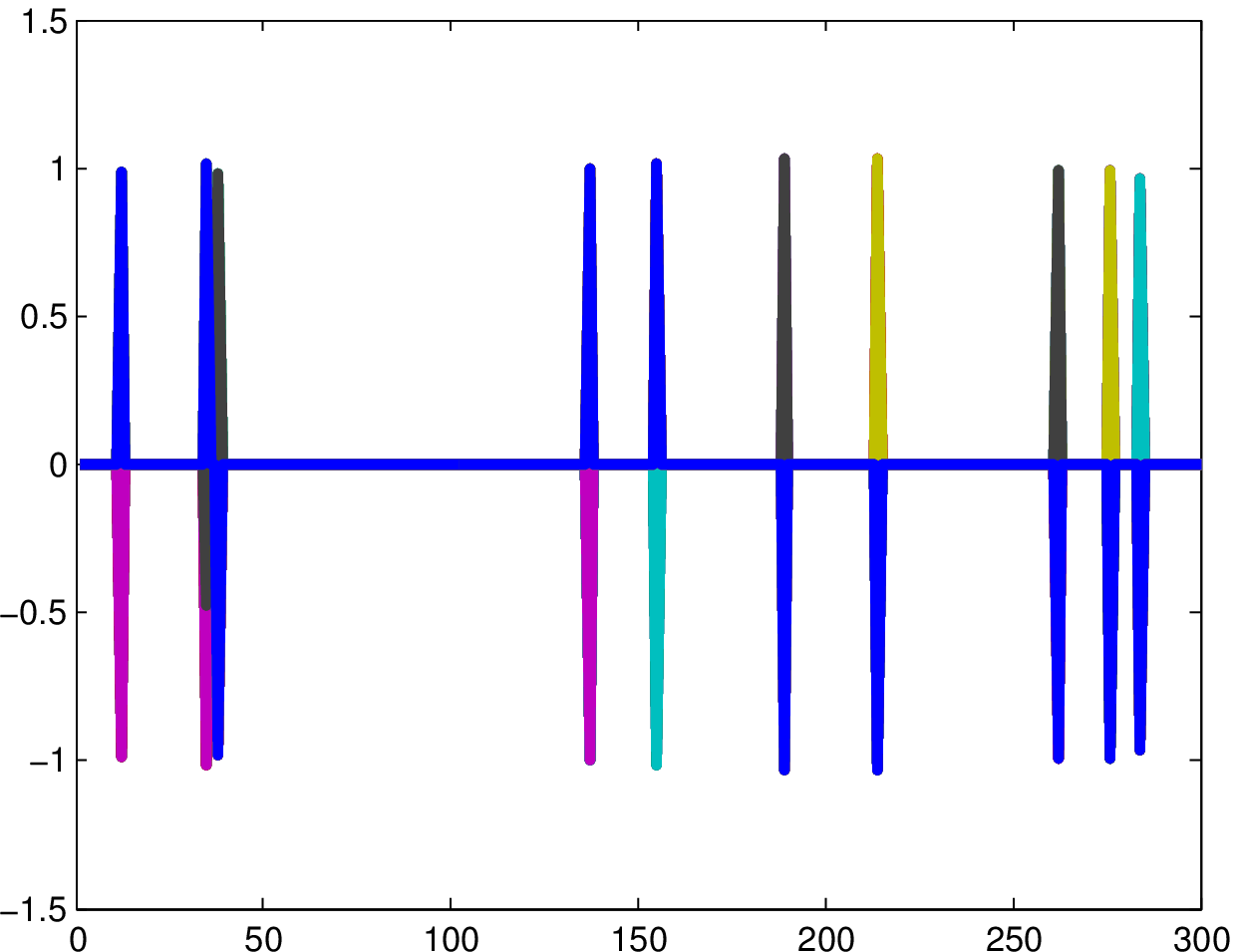}
\includegraphics[width=0.245\textwidth]{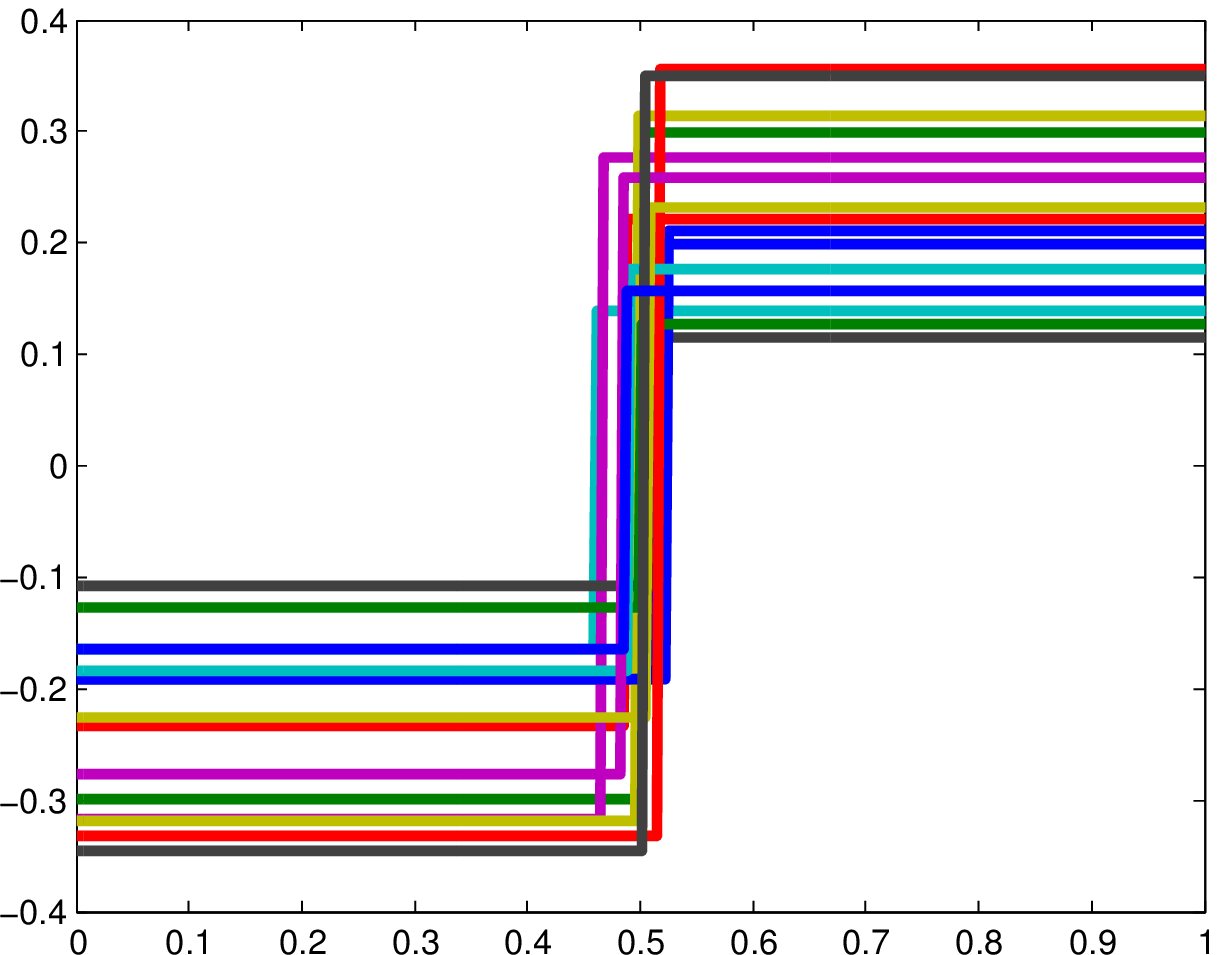}
\includegraphics[width=0.245\textwidth]{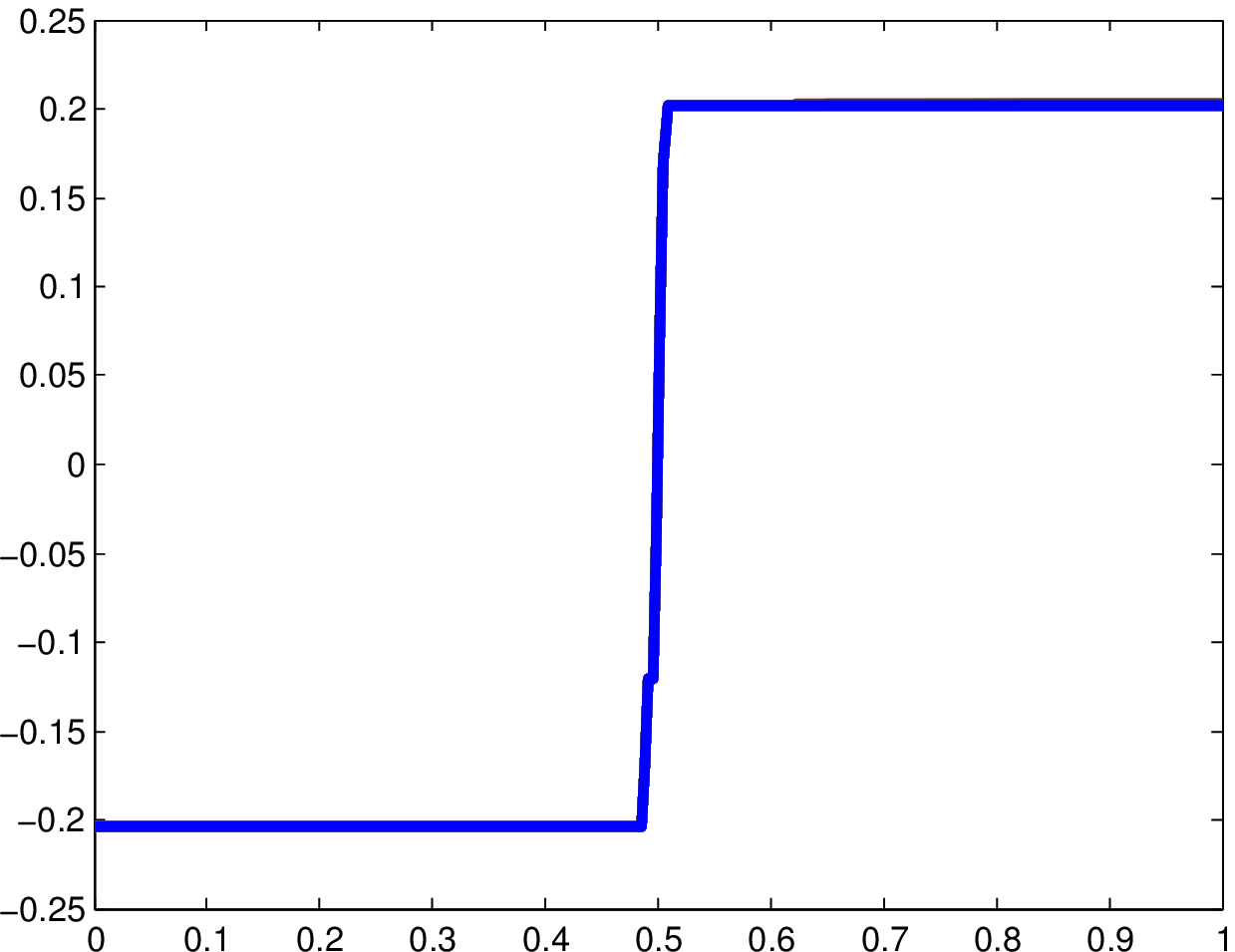}
\caption{Collaborative spectral analysis. Left: Fifteen different sparse input signals that have exactly ten peaks at common positions. Middle left: Ideal low pass filter in the $\|u\|_{\infty,1}$ sense -- as we can see the location of the common nonzero peaks is recovered. Middle right: Fifteen input signals that all have a jump around the position 0.5. Right: Ideal low pass filter in the $\|\nabla u\|_{\infty,1}$ sense -- all fifteen signals are approximated by one function which can be interpreted as the optimal single jump approximation of the input signals and hence naturally reflects the coarsest scale.}
\label{fig:collaboartaionExample}
\end{figure}

\vspace*{-18pt}

\section{Conclusion}
In this paper the theory of nonlinear spectral representation was generalized and extended in several ways.
The framework of \cite{Gilboa_spectv_SIAM_2014} was extended from the total variation functional to general one-homogeneous convex functionals.
Analogue transform representations were formulated based on variational regularization and on inverse-scale-space \cite{iss} (in addition to the original gradient flow formulation).
Moreover, an orthogonality of the decomposition is established and a new spectrum is suggested. Roughly, the spectrum measures the degree of activity of each scale. Here the spectrum is intrinsic to the functional space and can be viewed as a natural nonlinear extension of Parseval's identity.
Three examples of smoothing one-homogeneous functionals, other than TV, illustrate possible benefits of this approach for enhanced image and signal representation.

\vspace*{-6pt}

\section{Acknowledgements}

MB acknowledges support by ERC via Grant EU FP 7 - ERC Consolidator Grant 615216 LifeInverse. 

\bibliographystyle{splncs03}
\bibliography{refs}

\end{document}